\documentclass[a4paper,10pt]{article}
\usepackage[utf8]{inputenc}

\usepackage{amssymb,amsmath,amsthm}
\usepackage{verbatim}
\usepackage{centernot}
\usepackage{stmaryrd}
\usepackage{graphicx}
\usepackage{caption}
\usepackage{subcaption}

\newtheorem{thm}{Theorem}[section]
\newtheorem{thm*}{Theorem}
\newtheorem{lem}[thm]{Lemma}

\newtheorem{prop}[thm]{Proposition}
\newtheorem{prop*}{Proposition}
\newtheorem{cor}{Corollary}
\newtheorem{defn}{Definition}
\theoremstyle{definition}
\newtheorem{ex}{Example}[section]
\theoremstyle{remark}
\newtheorem*{rmk}{Remark}

\newcommand{\R}{\mathbb R}

\makeatletter
\newcommand{\xmapsto}[2][]{\ext@arrow 0599{\mapstofill@}{#1}{#2}}
\def\mapstofill@{\arrowfill@{\mapstochar\relbar}\relbar\rightarrow}
\makeatother
\title{}
\author{}

\begin{document}

\author{Roberta Guadagni}
\title{Non-decomposable cobordisms via Lagrangian moves}
\date{}

\maketitle

\begin{abstract}
We establish the use of Lagrangian diagram moves (cf \cite{lin}, \cite{datta}) to build Lagrangian cobordisms of Legendrian knots. We then apply the moves to explore how to invert decomposable cobordisms, after adding some stabilizations. 
As an application, we show an explicit construction of a Lagrangian concordance from a stabilized $m(9_{46})$ Legendrian knot to a stabilized unknot, i.e. a non-decomposable Lagrangian concordance of nonempty knots.
\end{abstract}

\section{Introduction}

Legendrian knots and links belong to a natural cobordism category, where the morphisms are \textit{Lagrangian cobordisms} (Definition \ref{def:Lcob}). While smooth cobordisms can (via Morse theory) be fully decribed by sequences of projection diagrams, the same does not hold for general Lagrangian cobordisms.
One of the most successful techniques to build Lagrangian cobordisms is via \textit{elementary} local moves in the front projection (Figure \ref{fig:leg_moves}). This produces a wide variety of Lagrangian cobordisms, called \textit{decomposable}.

Examples of cobordisms that are not (Hamiltonian isotopic to) decomposable cobordisms were elusive until recently. The first such examples were presented in \cite{lin} and were built using Lagrangian projections and Lagrangian moves (Figure \ref{fig:lagr_moves}). These examples required the top link to be empty, but hinted at the possibility of using Lagrangian moves to build more elaborate examples.

The first goal of this paper is to establish the use of Lagrangian moves to build Lagrangian cobordisms of Legendrians.  The resulting cobordisms may not be exact, so we expand our study to include non-exact ones, which we call \textit{weak Lagrangian cobordisms} (cf Definition \ref{def:Lcob_weak}).
Building on \cite{lin},\cite{datta}, we prove:

\begin{thm} (restated from Proposition \ref{thm:leg_extension})
If the Lagrangian projections of two Legendrian knots are joined by a sequence of Lagrangian moves as in Figure \ref{fig:lagr_moves}, the knots are weakly Lagrangian cobordant. 
\end{thm}

Sometimes exactness can be recovered, and the result upgraded to a genuine Lagrangian cobordism, see e.g. Corollary \ref{cor:lagr_conc}.

\begin{rmk}
Notice the following shortcoming: starting from the Lagrangian projection of a Legendrian, then performing Lagrangian moves, there is no guarantee a priori that the final diagram is the Lagrangian projection of a Legendrian.
\end{rmk}

\begin{figure}[!h]
    \begin{subfigure}{\textwidth}
        \centering
        \includegraphics[width=.7\linewidth]{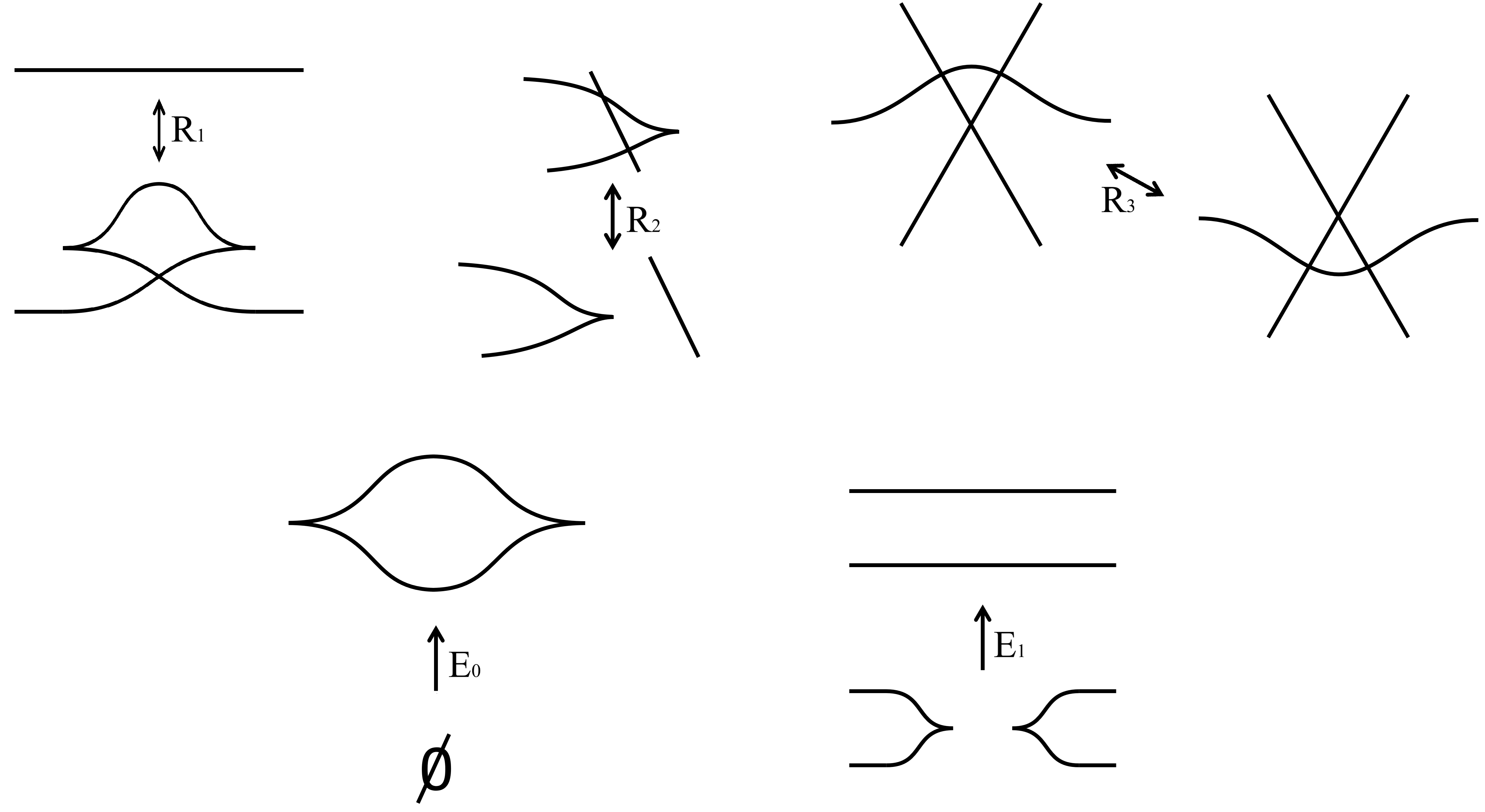}
        \caption{Decomposable moves in the front diagram.}
        \label{fig:leg_moves}
    \end{subfigure}
    \hfill

    \begin{subfigure}{\textwidth}
        \centering
        \includegraphics[width=.7\linewidth]{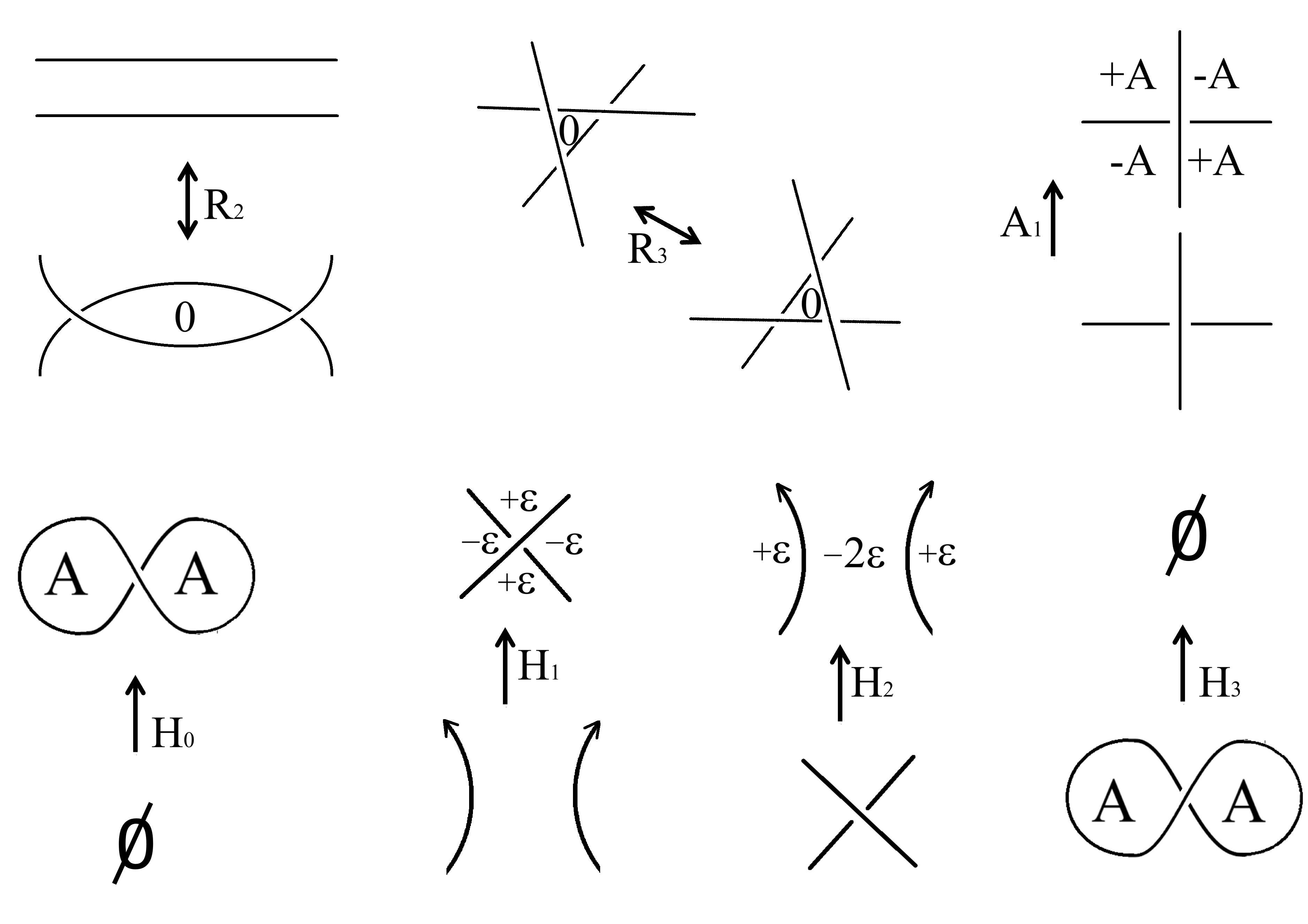}
        \caption{Lagrangian moves in the Lagrangian projection.}
        \label{fig:lagr_moves}
    \end{subfigure}
    \caption{Diagram moves.}
    \label{fig:moves}
\end{figure}

Recent work in \cite{Gol1}, \cite{Gol2} has produced more non-decomposable examples, with non-empty ends. Such examples rely on abstract $h$-principle arguments, so they don't provide an explicit construction. A core insight of their work: given a decomposable concordance of knots $C:K_1\prec K_2$, there must exist a Lagrangian concordance $C^r:S(K_2)\prec S(K_1)$ in the reverse direction, after adding sufficiently many positive and negative stabilizations (there is no estimate on the number of stabilizations needed) (cf \cite{rizell2024lagrangian}, \cite{Gol1}).

The present results stem from an effort to interpret the examples of \cite{Gol1}, \cite{Gol2} through the lens of Lagrangian moves. We specifically investigate how to invert a decomposable concordance, and more generally a given decomposable cobordism $C:L_1\prec L_2$, to obtain a Lagrangian cobordism $C^r:S(L_2)\prec S(L_1)$, using Lagrangian moves. 
We provide a set of tools that can be applied broadly. In particular, we use them to build the motivating example: a non-decomposable concordance in line with \cite{Gol1}.

\begin{thm}\label{thm:main_example}
There exists a Lagrangian concordance $$C:S_+^{4}S_-(m(9_{46}))\prec S_+^{4}S_-(\Upsilon)$$ built out of Lagrangian diagram moves, where $m(9_{46})$ is a Legendrian, max-tb mirror $9_{46}$ knot, $\Upsilon$ is the max-tb unknot, and $S_+^{4}S_-$ indicates four positive stabilizations and one negative.
\end{thm}

In particular, this bounds the number of necessary stabilizations (five). No claims are made on the sharpness of this bound.

The following open questions remain:
\begin{itemize}
\item can all decomposable cobordisms be inverted using these methods?
\item what is the minimum number of (positive/negative) stabilizations needed?
\end{itemize}





\textbf{Acknowledgements.} Thanks to Josh Sabloff for explaining the work in \cite{Gol1} and and asking the question. Thanks to Francesco Lin, Ipsita Datta, Georgios Dimitroglou Rizell, Roman Golovko for helpful conversations.

\section{Preliminaries}

This paper assumes familiarity with the fundamentals of contact geometry and Legendrian knot theory, as well as the construction of decomposable Lagrangian cobordisms of Legendrian knots as surveyed e.g. in \cite{survey}.
To ease the exposition, we will restrict to Legendrian links in the standard contact space $$\Lambda\subset(\R^3,\alpha=dz-ydx)$$
so that the front projection and the Lagrangian projection are defined globally by projecting onto the $xz$- and $xy$-plane respectively.
Below we explicitly review the needed material on Lagrangian diagrams and moves (cf \cite{lin}, \cite{datta}). 

\subsection{Clashing definitions of Lagrangian cobordisms}

Let us acknowledge an unfortunate overlap of definitions in the literature. 
For Legendrian links in contact $\R^3$ we have the following standard definition, rooted in the work of Arnol'd \cite{arn1},\cite{arn2} and subsequently \cite{egh2000sft}:
\begin{defn}\label{def:Lcob}
Given \textbf{Legendrian} links $\Lambda_1,\Lambda_2\subset (\R^3_{x,y,z},\alpha = dz-ydx)$, a \textbf{Lagrangian cobordism} $L:\Lambda_1\prec \Lambda_2$ is an \textbf{exact} orientable Lagrangian submanifold of the symplectization $$L\subset (\R\times\R^3,\omega=d(e^t\alpha))$$ that admits a pair of real numbers $T_\pm$ satisfying:
\begin{itemize}
\item $L |_{t\in [T_-,T_+]}$ is compact ($[T_-,T_+]$ is the \textbf{support} of the cobordism),
\item $L |_{t\in (-\infty,T_-]} = (-\infty,T_-] \times \Lambda_1$ (\textbf{cylindrical lower end}), 
\item $L |_{t\in [T_+,\infty)}= [T_+,\infty) \times \Lambda_2$ (\textbf{cylindrical upper end})
\end{itemize}
and the primitive of $e^t\alpha$ along $L$ can be chosen to be simultaneously constant for $t = T_-$ and for $t = T_+$ (\textbf{leveledness} property);
\end{defn}
 Leveledness is automatically true for cobordisms of knots. In general it can be omitted (thus defining \textit{unleveled} cobordisms), but we don't pursue this here. Similarly, we omit discussing non-orientable versions.

Within symplectic $\R^4$, cobordisms (of smooth links) that happen to be Lagrangian are sometimes also called `Lagrangian cobordisms' (cf \cite{lin},\cite{datta}), however we rename them for clarity (below, $\pitchfork$ denotes transversality): 
\begin{defn}\label{def:LcobR4} (cf \cite{lin},\cite{datta})
Given links $K_1,K_2\subset \R^3_{v,s,u}$, an \textbf{L-cobordism}  $L:K_1 \triangleleft K_2$ is any properly embedded, oriented, compact Lagrangian submanifold $$L\subset(\R^4_{t,v,s,u},\omega=dt\wedge dv + ds\wedge du)$$ such that:
\begin{itemize}
\item  $L\subset\{a\leq t\leq b\}$ for some $a<b$ ($[a,b]$ is the \textbf{support} of $L$);
\item $\partial L = (L\cap\{t=a\}) \cup (L\cap\{t=b\})$, $L \pitchfork \{t=a\}$, $L \pitchfork\{t=b\}$;
\item $\partial_-L:= L\cap\{t=a\} = K_1\times\{a\} $ (\textbf{lower boundary} $K_1$)
\item $\partial_+L:= L\cap\{t=b\}= K_2\times\{b\}$ (\textbf{upper boundary} $K_2$).
\end{itemize}
\end{defn}

We make a different choice of coordinates compared to \cite{datta}. This aligns the appearance of the moves with the diagrams in \cite{lin}, so that the diagrams can later be compared to the standard Lagrangian projection of Legendrian knots, with the moves happening in the positive symplectization direction.

Definitions \ref{def:Lcob}, \ref{def:LcobR4} can clash, since exactness is assumed in \ref{def:Lcob} and not \ref{def:LcobR4}. We thus add the \textbf{weak} qualifier to Lagrangian cobordisms that potentially fail to be exact (the leveledness property is then also dropped):
\begin{defn}\label{def:Lcob_weak}
Given Legendrian links $\Lambda_1,\Lambda_2\subset (\R^3,\alpha = dz-ydx)$, a \textbf{weak Lagrangian cobordism} $L:\Lambda_1\prec^w \Lambda_2$ is an orientable Lagrangian submanifold (not necessarily exact) $L\subset (\R\times\R^3,\omega=d(e^t\alpha))$ that admits a pair of real numbers $T_\pm$ satisfying:
\begin{itemize}
\item $L |_{t\in [T_-,T_+]}$ is compact ($[T_-,T_+]$ is the \textbf{support} of the cobordism),
\item $L |_{t\in (-\infty,T_-]} = (-\infty,T_-] \times \Lambda_1$ (\textbf{cylindrical lower end}), and
\item $L |_{t\in [T_+,\infty)}= [T_+,\infty) \times \Lambda_2$ (\textbf{cylindrical upper end}).
\end{itemize}
\end{defn}

There are other subtle differences between Definitions \ref{def:Lcob}, \ref{def:LcobR4}. In particular, Definition \ref{def:LcobR4} yields a non-reflexive relation: in general, $K\ntriangleleft K$ (cf \cite{datta}); Definition \ref{def:Lcob} yields a reflexive relation:  $\Lambda\prec_{\R\times\Lambda}\Lambda$ (more generally, Legendrian isotopic links always admit a Lagrangian cobordism in the symplectization, see \cite{chantraine2010concordance}).

\subsection{Lagrangian diagrams and moves}

All fundamental diagram moves are shown in Figure \ref{fig:moves}. 

The decomposable elementary moves are: Reidemeister $R_1,R_2,R_3$; $0$-handle $E_0$; $1$-handle $E_1$ (cf \cite{survey}). 
For the purpose of building decomposable cobordisms, the only relevant data in a front projection is its smooth topology.

Lagrangian moves behave differently: here it is important to track the areas enclosed by the diagram as well as its smooth topology.
\begin{defn}(cf \cite{lin,datta})
A \textbf {Lagrangian diagram} $D$ is given by the data:
\begin{itemize}
\item a regular link diagram $\Delta \subset\R^2_{x,y}$;
\item a function $\mathcal{A}:\mathcal{D}(\Delta)\rightarrow\R_{\geq 0}$ where $\mathcal{D}(\Delta)$ is the set of disks in $\R^2_{xy}\backslash \Delta$.
\end{itemize}
Given a smooth link $K\subset\R^3_{x,y,z}$ (with regular $xy$ projection), the \textbf{Lagrangian diagram} $D_K$ \textbf{associated to} $K$ is given by:
\begin{itemize}
\item the $xy$-projection of $K$, $\pi_{xy}K$;
\item the function $\mathcal{A}:\mathcal{D}(\pi_{xy}K)\rightarrow\R$ associating to each disk $\delta$ its ($dx\wedge dy$)-area.
\end{itemize}
A diagram is  \textbf{Legendrian} if $D=D_\Lambda$ for a Legendrian $\Lambda$.
Two diagrams are \textbf{equivalent} when they are isotopic in $\R^2$ and $\mathcal{A}$ takes identical values on corresponding disks.
\end{defn}
Notice that any abstract Lagrangian diagram $D$ with strictly positive areas ($\mathcal{A}>0$) can be tweaked in $\R^2$ to find a smooth link $K$ such that $D\cong D_K$. Lemma \ref{lem:leg_diag} later implies that equivalent Legendrian diagrams correspond to Legendrian isotopic knots. 
\begin{figure}[htbp]
    \centering
    \makebox[\textwidth][c]{%
        \includegraphics[width=.7\linewidth]{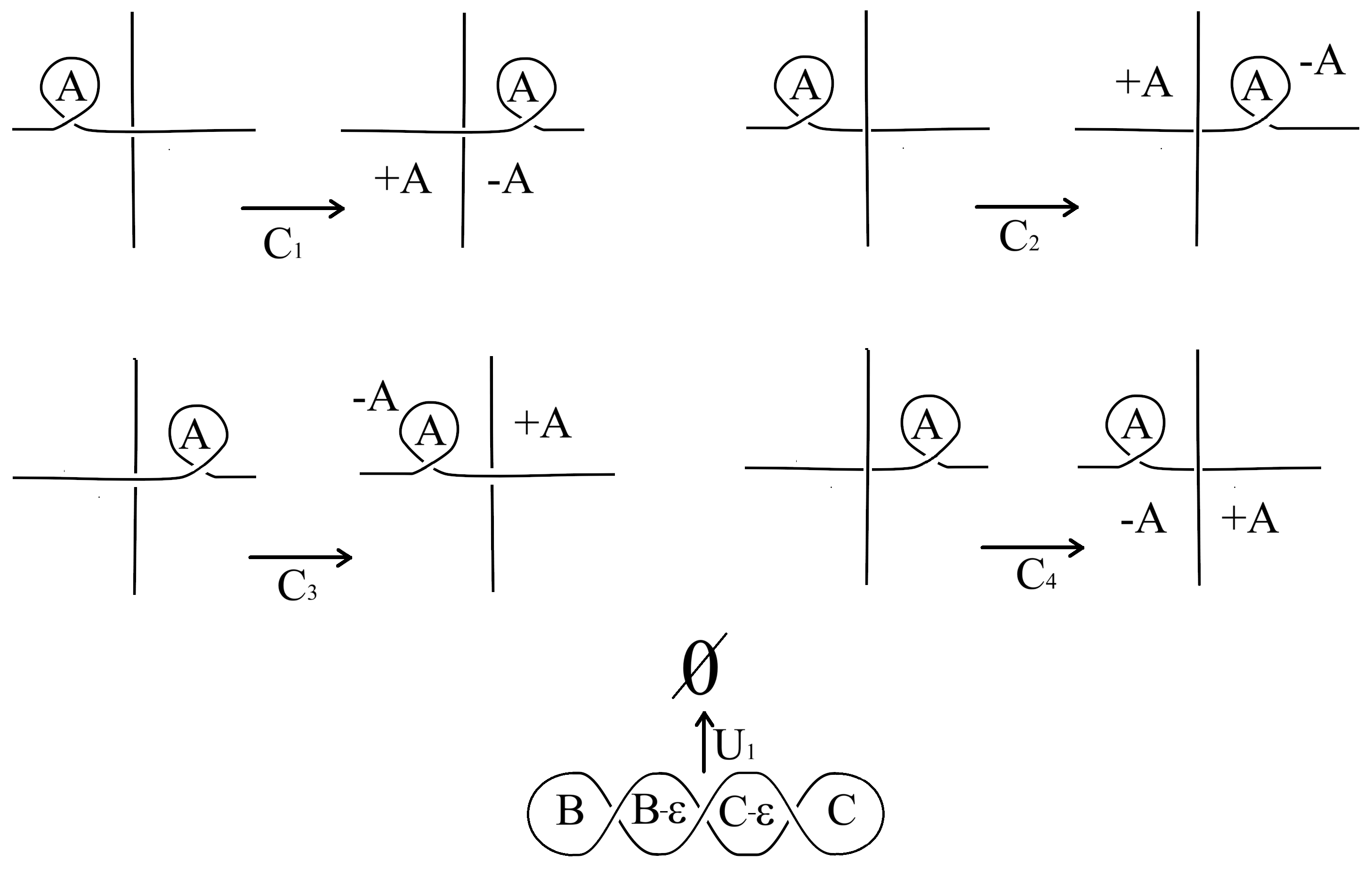}
    }
    \caption{Curl moves and unknot trick ($B>C\geq\varepsilon\geq 0$)}
    \label{fig:curls}
\end{figure}

Lagrangian moves are defined between (equivalence classes of) Lagrangian diagrams.
The fundamental Lagrangian moves are shown in Figure \ref{fig:lagr_moves} and include: Reidemeister $R_2,R_3$ (no $R_1$); area change $A_1$; $0$-handle $H_0$; $1$-handles $H_1,H_2$; $2$-handle $H_3$. The notation $+A,-A$ refers to increases/decreases in the areas; in particular, a Lagrangian move decreasing an area by an amount $A$ can only be applied if the current area is at least $A$. We allow areas to be exactly $0$ in the formal diagrams, so the moves can be applied literally (the assumption being that the areas surrounding a $0$ are themselves non-zero).


We will rely on more moves, derived from the fundamental ones. Figure \ref{fig:curls} presents the curl moves $C_{1-4}$ and the `unknot trick' $U_1$, which only applies if $B>C$. Curl moves can be built from $R_2,R_3,A_1$ while $U_1$ is built from $R_2,A_1$ plus one cap $H_3$ (cf \cite{lin}).
In Section \ref{sec:areas} we will further expand this toolbox.
\begin{rmk}
Pay attention to how the move $A_1$ is strictly directional and can't be reversed. As a consequence, all curl moves are strictly directional as well.
\end{rmk}

The following is directly derived from Proposition 3.10 of \cite{datta}:
\begin{prop}\label{prop:lagr_move}(\cite{datta}, Proposition 3.10)
Given links $K_1,K_2\subset\R^3$, if $D_{K_2}$ can be obtained from $D_{K_1}$ via a move from Figure \ref{fig:lagr_moves}, plus area moves including a non-zero area move around each crossing of $K_2$, there exists an L-cobordism $L:K_1\triangleleft K_2$.
\end{prop}

\begin{proof}
More precisely, this is a corollary of \cite{datta}, Proposition 3.10. We can check that if $D_{K_2}$ differs from $D_{K_1}$ by a move in Figure \ref{fig:lagr_moves} plus area moves, then all the hypotheses of \cite{datta}, Proposition 3.10 are satisfied. Indeed
\begin{itemize}
\item total area is preserved since each move in Figure \ref{fig:lagr_moves} preserves it;
\item all areas of $D_{K_1},D_{K_2}$ are strictly positive;
\item if a disk only has positive (negative) corners, then a sequence of area moves will strictly increase (decrease) its area;
\item if the sequence from $D_{K_1}$ to $D_{K_2}$ includes one move $M$ from Figure \ref{fig:lagr_moves}, it equivalently includes one move $A_0\circ M\circ A_0$ from the list in \cite{datta}.
\end{itemize}
\end{proof}

\begin{rmk}
The results in \cite{lin} stop short of proving the above proposition. Lagrangian moves are built in \cite{lin} by leveraging the work of \cite{Sau04}; a straightforward concatenation of such moves builds an L-cobordism connecting $K_1$ to some link $K_3$ with $D_{K_3}=D_{K_2}$, but in general $K_2\neq K_3$ (there is little control on the $z$-coordinate); this is of course irrelevant if $K_2=\emptyset$.
The stronger version from \cite{datta} above follows from a ``wiggling'' argument that allows control of the $z$-coordinate.
\end{rmk}

We will also need a gluing lemma (cf \cite{datta}, Lemma 2.4):

\begin{lem}\label{lem:glueing}(cf \cite{ELST08},\cite{datta})
Let $L_1,L_2\subset\R^4$ be L-cobordisms, with support $[a,b]$ and $[b,c]$ respectively. Assume $\partial_+L_1 = \partial_-L_2$. Then, for any $\varepsilon>0$ small enough, there exists an L-cobordism $L_2\circ L_1$ supported in $[a,c]$, coinciding with $L_1$ for $t<b-\varepsilon$ and with $L_2$ for $t>b+\varepsilon$. Topologically, $L_2\circ L_1$ is homeomorphic to $L_2\cup L_1$.
\end{lem} 

The topology statement is not explicitly included in Lemma 2.4 of \cite{datta}, but is implied by the same proof.  
Thanks to this, L-cobordisms can be composed. 


\subsection{Lagrangian diagrams of Legendrians}

Given the Lagrangian diagram of a knot $D=D_K\subset \R^2_{x,y}$ with total signed area $0$, we can recover a Legendrian $\Lambda$ by choosing a basepoint $p\in D$ and letting $$ z(q) = z_p + \int_p^q ydx \text{\ \ for all\ \ \ } q\in K. $$

Then $D=D_\Lambda$ (i.e., $D$ is Legendrian) if and only if: at each crossing $c\in D$, $z$ takes a higher value on the top strand compared to the lower one, i.e., the area enclosed by any arc starting at the lower strand and ending at the upper strand is positive (cf \cite{datta}; links behave similarly but have more tricky basepoint choices). This property is stable with respect to small changes in the area function, i.e.:

\begin{lem}\label{lem:leg_diag}
Let $\Lambda$ be a Legendrian knot. If $\varepsilon>0$ is small enough, then any Lagrangian diagram $D=D_K$ such that:
\begin{itemize}
\item $D$ has total area $0$;
\item $D_\Lambda,D$ are isotopic as diagrams in $R^2$;
\item the area functions $\mathcal{A}_D$, $\mathcal{A}_{D_\Lambda}$ differ by less than $\varepsilon$ on each disk;
\end{itemize}
is Legendrian: $D=D_{\Lambda_2}$.
Also, $\Lambda,\Lambda_2$ are Legendrian isotopic.
\end{lem}

\begin{proof}
At each crossing $c\in D_\Lambda$, let $\Delta z(c)=z_{t}(c)-z_{b}(c)$ be the height difference from the top to the bottom strand. Denote by $\mathcal{H}_{D_\Lambda}=\{\Delta z(c) \}$ the set of all such heights. For each crossing $c$ consider the oriented arc starting at the lower strand at $c$ and ending at the upper strand. Let $n_c$ be the number of disks that compose the area enclosed by the arc (counted with multiplicity given by the absolute value of the winding number on each disk). Let $\varepsilon <\frac{\min (\mathcal{H}_{D_\Lambda})}{\max_c n_c}$.
If $D$ has total area $0$, pick a basepoint and lift $D$ to $\Lambda_2$ via $$ \zeta(q) = \zeta_p + \int_p^q ydx \text{\ \ for all\ \ \ } q\in K. $$
We need to check that, at each crossing, the total area $\alpha_c$ enclosed by the arc starting at the lower strand at $c$ and ending at the upper strand is positive. The area $\alpha_c$ is given by a signed count of $\mathcal{A}_D(\delta)$ over $n_{c}$ disks. Since $|\mathcal{A}_{D_\Lambda}-\mathcal{A}_{D}|<\varepsilon$, we have $|\alpha_c-\Delta z(c)|<n_{c}\varepsilon <\min (\mathcal{H}_{D_\Lambda})$, so $\alpha_c>0$ as needed.
Consider the isotopy $\{D_\tau\}_{0\leq\tau\leq 1}$ from $D$ to $D_\Lambda$; let $\mathcal{A}_\tau=(1-\tau)\mathcal{A}+\tau\mathcal{A}_\Lambda$. Then $D_\tau$ still satisfies $\alpha_c>0$ for all crossings $c\in {D_\tau}$. Lift each $D_\tau$ to the Legendrian $\Lambda_\tau$. If the base point $p$ and value $z_p$ are chosen smoothly with respect to $\tau$, $\{\Lambda_\tau\}_{0\leq\tau\leq 1}$ is a Legendrian isotopy from $\Lambda_2$ to $\Lambda$.
\end{proof}

\section{Lagrangian cobordisms from Lagrangian moves}

Lagrangian moves induce L-cobordisms (cf Proposition \ref{prop:lagr_move}). We need to bridge the gap to obtain Lagrangian cobordisms in the symplectization.
We will show:

\begin{prop}\label{thm:leg_extension}
Let $\Lambda_1,\Lambda_2\subset \R^3$ be Legendrian knots such that the Lagrangian diagram $D_{\Lambda_2}$ can be obtained from $D_{\Lambda_1}$ via a sequence of orientation-compatible Lagrangian moves as in Figure \ref{fig:lagr_moves}. There exists a weak Lagrangian cobordism $L:\Lambda_1\prec^w\Lambda_2$, with the smooth topology specified by the sequence of moves.
\end{prop}

For concordances, we can recover exactness from the underlying topology, and thus drop the \textit{weak} qualifier:
\begin{cor}\label{cor:lagr_conc}
Let $\Lambda_1,\Lambda_2\subset \R^3$ be Legendrian knots such that $D_{\Lambda_2}$ can be obtained from $D_{\Lambda_1}$ via Lagrangian moves. If the underlying smooth cobordism from $\Lambda_1$ to $\Lambda_2$ is a concordance, there exists a Lagrangian concordance $\Lambda_1\prec\Lambda_2$. 
\end{cor}

\begin{proof}[Proof of corollary]
Apply Proposition \ref{thm:leg_extension} to build a weak Lagrangian cobordism $L:\Lambda_1\prec^w\Lambda_2$.  
Since $L$ is a concordance, its fundamental group is generated by $\Lambda_1$ (or $\Lambda_2$). Since  $\Lambda_{1,2}$ are Legendrian, $e^t\alpha$ does not have monodromy around them, thus $L$ is exact. Given exactness, the leveledness property is well-defined. It is also trivially true since $\Lambda_{1,2}$ are knots.
\end{proof}

Before proving Proposition \ref{thm:leg_extension}, we delve into the correspondence between the standard symplectic space $\R^4$ and the symplectization space $\R\times \R^3$.
There are multiple ways to identify these spaces; we choose: $$\beta:\R_t\times\R^3_{x,y,z}\rightarrow \R^4_{t,v,s,u}\ \ \ \ \ \ \ \beta(t,x,y,z)=(t, e^tz, x, e^ty).$$ 
Our choice differs slightly from the choices of both \cite{lin} and \cite{datta}, but it is the most natural for our purpose. Restricting to $t$-slices we have: 
\begin{equation}\label{contactomorphic}
\beta|_{t=a}:\{a\}\times\R^3_{x,y,z}\rightarrow \{a\}\times\R^3_{v,s,u} \text{\ \ \ \ \ \ } (\beta^{-1})^*{\alpha} = e^{-a}(dv-uds).
\end{equation}

This behaves nicely with respect to the Legendrian property. Define:
\begin{defn}
Given a Legendrian link $\Lambda\subset(\R^3_{x,y,z},dz-ydx)$, and $\varepsilon\in\R$, the \textbf{$\varepsilon$-expansion} of $\Lambda$ is the link $\Lambda^\varepsilon=\{(x,e^\varepsilon y,e^\varepsilon z)| (x,y,z)\in\Lambda\}$.
\end{defn}
Given $\Lambda$, all its $\varepsilon$-expansions are Legendrian isotopic (for any $\tau$, $d(e^\tau z)-e^\tau y dx = e^\tau (dz-ydx) =0$, so $\{\Lambda^\tau\}_{\varepsilon_1\leq \tau\leq \varepsilon_2}$ is a Legendrian isotopy). We have:
\begin{lem}\label{lem:cobordisms}
\begin{enumerate}
\item The map $\beta$ is an exact symplectomorphism;
\item If $L:\Lambda_1\prec\Lambda_2$ is a Lagrangian cobordism with support $[t_1,t_2]$, then $\beta(L|_{[t_1,t_2]})$ is an L-cobordism $\Lambda_1^{t_1}\triangleleft\Lambda_2^{t_2}$;
\item If $\mathcal{L}:K_1\triangleleft K_2$ is an L-cobordism supported on $[a,b]$, and $K_1,K_2$ are Legendrians in $(\R^3,dv-uds)$, then $\beta^{-1}(\mathcal{L})$ yields a weak Lagrangian cobordism $K_1^{-a}\prec^w K_2^{-b}$ supported on $[a-\varepsilon,b+\varepsilon]$, for any $\varepsilon>0$.
\end{enumerate}
\end{lem}

\begin{proof}
1. Fix a primitive $\lambda=-vdt-uds$ on $\R^4_{t,v,s,u}$ (the choice of $\lambda$ doesn't matter, since $d\lambda_1=d\lambda_2 =\omega\Rightarrow \lambda_1-\lambda_2 = d\theta$ for any two choices in $\R^4$). Then $$\beta^*(\lambda)=
-e^tzdt-e^tydx = e^t(\alpha)-d(e^tz).$$

2. Let $L$ be a Lagrangian cobordism; the cylindrical ends ensure that $L\pitchfork \{t=t_1,t_2\}$; the Lagrangian property and transversality are preserved by $\beta$, so $\beta(L|_{[t_1,t_2]})$ is an L-cobordism and $\beta(\{t_1\}\times \Lambda_1) = \Lambda_1^{t_1}$, $\beta(\{t_2\}\times \Lambda_2) = \Lambda_2^{t_2}$.

3. Let $\mathcal L$ be an L-cobordism supported in $[a,b]$; consider $\beta([a-\varepsilon,a]\times K_1^{-a}),  \beta([b,b+\varepsilon]\times K_2^{-b})$: these are L-cobordisms as well (see part 2). They agree with $\mathcal{L}$ for $t=a,b$, so the gluing Lemma \ref{lem:glueing} yields $\bar{\mathcal L}=\beta([b,b+\varepsilon]\times K_2^{-b})\circ\mathcal{L}\circ\beta([a-\varepsilon,a]\times K_1^{-a})$. Then, $\beta^{-1}(\bar{\mathcal{L}})$ is Lagrangian and has cylindrical collars, thus it is a weak Lagrangian cobordism $K_1^{-a}\prec_w K_2^{-b}$.
\end{proof}

We are ready to prove the proposition.
\begin{proof}[Proof of Proposition \ref{thm:leg_extension}]
Assume that $D_{\Lambda_2}$ can be obtained from $D_{\Lambda_1}$ via a sequence $\mathcal{S}$ of $n$ Lagrangian moves. Proposition \ref{prop:lagr_move} can't be applied directly, as it requires additional area moves.
We will then tweak the diagrams in the following way. Pick $\varepsilon$ small enough for Lemma \ref{lem:leg_diag} to hold for $D_{\Lambda_2}$. Starting at $D_{\Lambda_1}$, apply the first move in the sequence $\mathcal{S}$. Let $C$ be the resulting number of crossings; choose $\delta>0$, infinitesimal with respect to: any non-zero area, non-zero difference of areas in the diagram, any area parameters for $H_{1,2}$ in the sequence $\mathcal{S}$, and the quantity $\frac{\varepsilon}{nC}$.
Apply an area move of weight $\delta$ at each crossing, to obtain a new diagram $D_{\Lambda'_1}$ (with non-zero areas). Repeat the procedure for the next move, starting from $D_{\Lambda_1'}$. Repeat: for each move in $\mathcal{S}$, choose an appropriate $\delta$ and apply more area moves.
We need to check that adding these area moves does not disrupt the sequence, i.e. area equalities are preserved when needed. Let $H$ be a move (of type $H_1$ or $H_2$) that separates two diagram components in such a way that the subsequent total areas are $B_1,B_2$ on the two resulting components. Since $\delta$'s are infinitesimal, the total areas on the two sides are almost unchanged after adding the extra area moves; the area parameter in $H$ can then be adjusted to still obtain exactly total areas $B_1,B_2$. This ensures that subsequent caps can be applied, since the total area of each component can be controlled exactly.

The final resulting diagram $D$ coincides with $D_{\Lambda_2}$ in $\R^2$, and the choice of $\delta$ appropriately small ensures $|\mathcal{A}_D-\mathcal{A}_{D_{\Lambda_2}}|<\varepsilon$ on each disk. Therefore, via Lemma \ref{lem:leg_diag}, $D=D_{\Lambda_3}$ for some $\Lambda_3$ Legendrian isotopic to $\Lambda_2$.
Apply Proposition \ref{prop:lagr_move} to each Lagrangian move and corresponding area moves. The resulting L-cobordisms can be concatenated (via Lemma \ref{lem:glueing}) into $L:\Lambda_1\triangleleft \Lambda_3$.
Let $[a,b]$ be the support of $L$. According to Lemma \ref{lem:cobordisms}(3), there exists a weak Lagrangian cobordism $\Lambda_1^{-a}\prec^w \Lambda_3^{-b}$.
Since $\Lambda_1^{-a},\Lambda_1$ are isotopic, and so are $\Lambda_3^{-b},\Lambda_3,\Lambda_2$, we get $L:\Lambda_1\prec^w\Lambda_2$ (Legendrian isotopic knots are Lagrangian concordant, cf \cite{chantraine2010concordance}).

The smooth topology of $L$ can then be recovered from the individual moves, thanks to Lemma \ref{lem:glueing}.
\end{proof}

\section{Diagram conversion}\label{sec:front_diagrams}

A front diagram can always be converted into a Lagrangian diagram: evaluate $y=\frac{dz}{dx}$, plot the result in the $xy$-plane, compute the areas. This can be cumbersome. There is a known shortcut to obtain a Lagrangian projection diagram from a front: smooth all left cusps, add a loop to each right cusp (cf \cite{ng2003computable}). This method, however, does not accurately represent the areas, so it won't be useful here.
We instead use the following (again, cf \cite{ng2003computable}):
\begin{defn}\label{def:front_diagrams}
A \textbf{linear front diagram} is a piecewise linear curve in $\R^2$ such that:
\begin{itemize}
\item vertices appear at distinct $x$ values;
\item each segment has a distinct, non-vertical slope.
\end{itemize}
\end{defn}

\begin{figure}[htbp]
    \centering
    \makebox[\textwidth][c]{%
        \includegraphics[width=.8\linewidth]{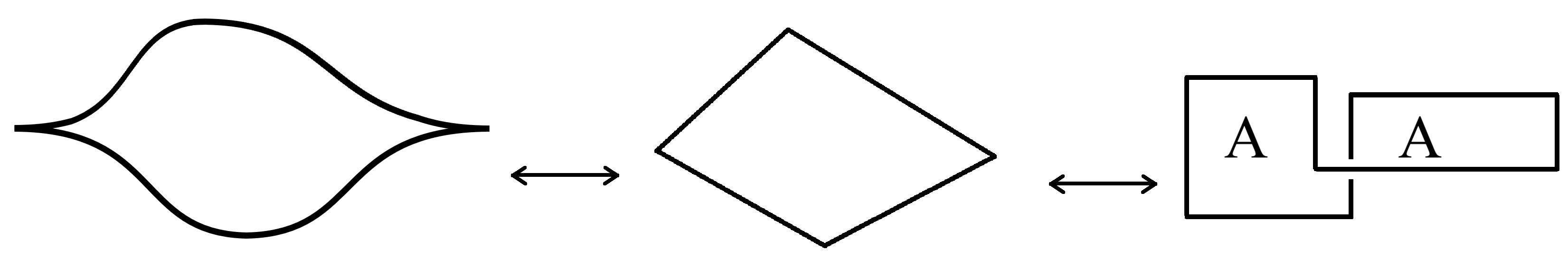}
    }
    \caption{Linearized front and corresponding Lagrangian diagram of an unknot.}
    \label{fig:linear_unknot}
\end{figure}

Given a front diagram, we can always find piecewise linear approximations. Conversely, given a linear front, we can recover a front diagram by applying smoothings close to each vertex, so that the tangent is horizontal at the vertices.

Linear front diagrams are then easily converted into Lagrangian diagrams: all Reeb chords appear at the vertices, and all areas are built out of rectangles. 
See Figure \ref{fig:linear_unknot} for a simple example.
A linear front diagram of the max-tb Legendrian $m(9_{46})$ knot, with its Lagrangian conversion, is shown in Figure \ref{fig:946}.





\subsection{Stabilizations}\label{sec:stabilizations}

\begin{figure}[!htbp]
    \centering
    \makebox[\textwidth][c]{%
        \includegraphics[width=1\linewidth]{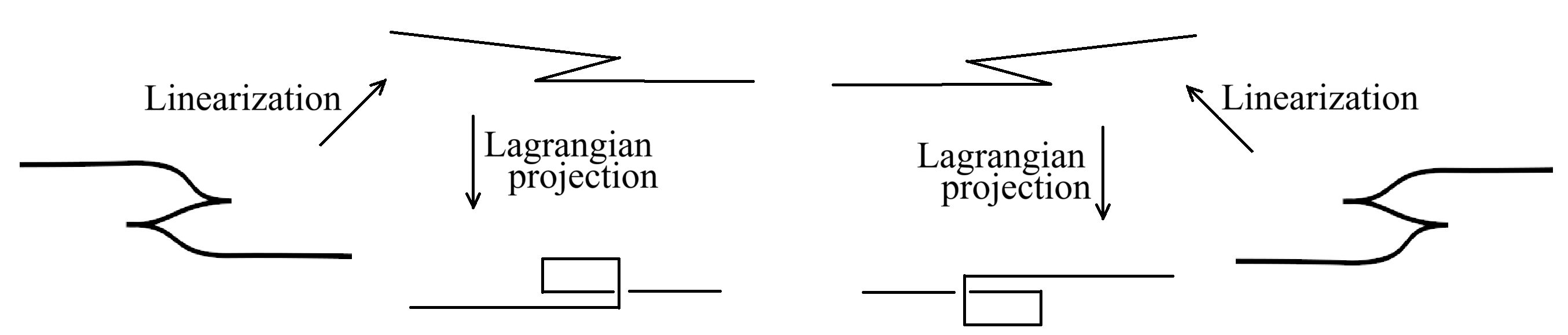}
    }
    \caption{Stabilizations.}
    \label{fig:stabilizations}
\end{figure}
As a warmup, we apply linear diagrams to stabilizations. Figure \ref{fig:stabilizations} shows that stabilizations correspond to loops: a positive stabilization adds a loop on the right side of the oriented strand, a negative stabilization on the left side.
The loop always has a negative crossing. In fact, it is not possible to add a positive loop while preserving the Legendrian property (cf \cite{datta}).
Stabilizations can be arbitrarily small, so a Lagrangian diagram can be formally stabilized by adding a loop with area $0$.

\section{Inverting decomposable moves}\label{sec:inversions}

We now explore the question: given a decomposable Lagrangian cobordism $L$, can we build `inverse' moves to obtain a Lagrangian cobordism in the opposite $t$ direction, smoothly isotopic to $-L$?

Legendrian Reidemeister moves can be inverted with respect to the symplectization coordinate. Elementary handle attachments, on the other hand, happen in a specific $t$ direction.
The goal is thus to invert decomposable handles via (decomposable or Lagrangian) moves connecting the links in the opposite direction. We focus on two main study cases: $1$-handles within a decomposition that doesn't include $0$-handles, or $1$-handles within a concordance. The general case can then be approached with a blend of these techniques.

\subsection{Cobordisms without $0$-handles}
\begin{figure}[!h]
        \centering
        \includegraphics[width=.5\linewidth]{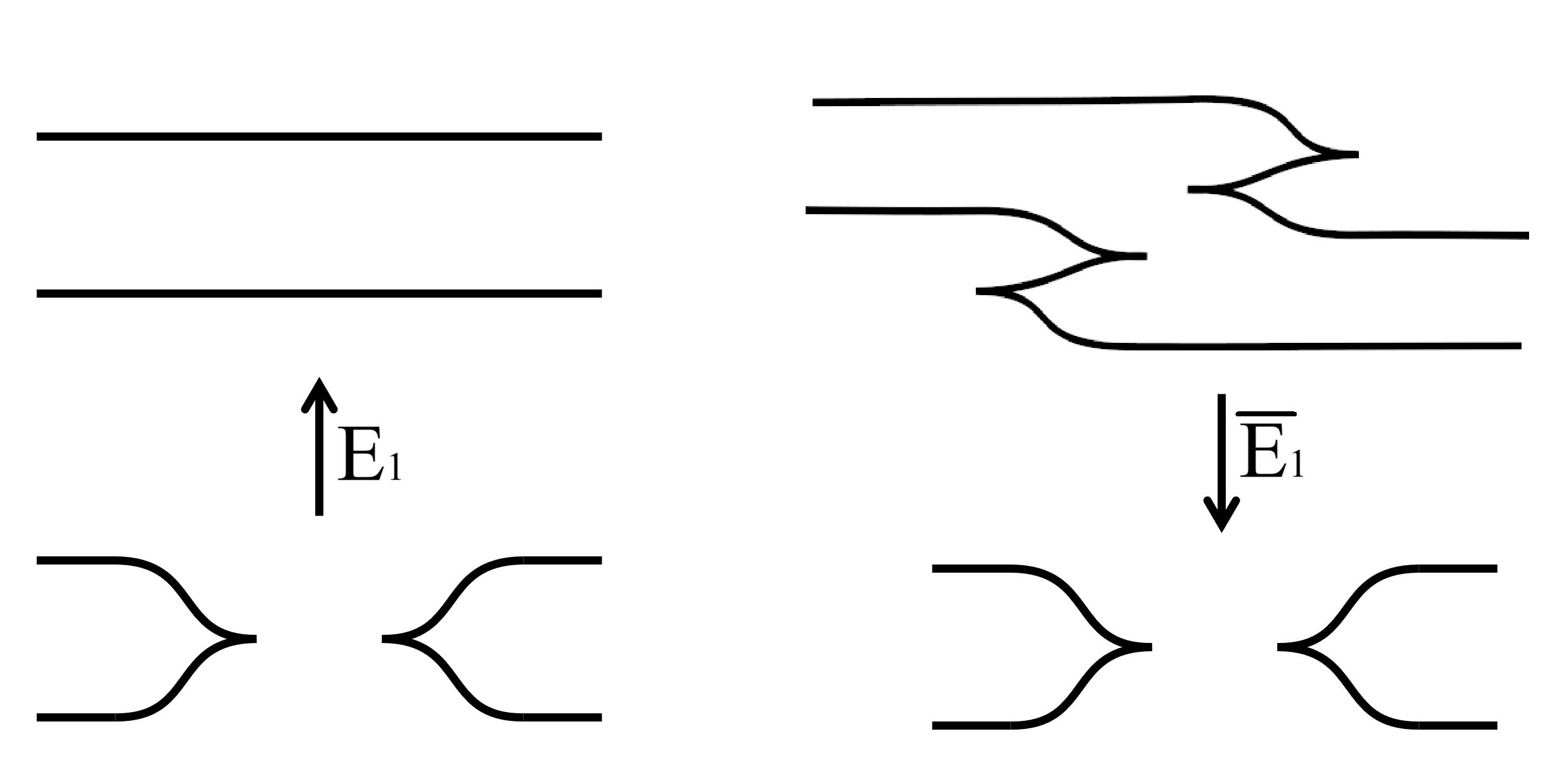}
    \caption{A decomposable $1$-handle attachment in the reverse direction of $E_1$.}
    \label{fig:E1_inv}
\end{figure}

Let $L:\Lambda_1\prec \Lambda_2$ be a decomposable cobordism. If $L^r: {\Lambda_2}^r\prec \Lambda_1$ is isotopic to $-L$, the Thurston-Bennequin numbers must satisfy: $$tb(\Lambda_1) - tb({\Lambda_2}^r) = -\chi(L^r) = -\chi(L) = tb(\Lambda_2) - tb(\Lambda_1).$$
 Thus $tb(\Lambda_2)-tb({\Lambda_2}^r) = -2\chi(L)$, i.e. we need to add $-2\chi(L)$ stabilizations to $\Lambda_2$ (specifically, $-\chi(L)$ positive and $-\chi(L)$ negative stabilizations, to preserve total rotation).
In the absence of $0$-handles, this is enough:
\begin{lem}
Let $L:\Lambda_1\prec \Lambda_2$ be a decomposable cobordism that admits a decomposition without $0$-handles. Then there is a decomposable cobordism 
$$L^r: {S_+}^n({S_-}^n(\Lambda_2))\prec \Lambda_1$$
where $n=-\chi(L)$ and $L^r$ is smoothly isotopic to $-L$.
\end{lem}

\begin{proof}
Follow all the isotopies of $L$ in reverse, and replace each handle attachment $E_1$ with the decomposable move $\bar{E_1}$ shown in Figure \ref{fig:E1_inv}.
\end{proof}

\subsection{Concordances}\label{sec:conc}
Things are trickier when $0$-handle attachments are present. We ask: given a decomposable concordance $\Lambda_1\prec \Lambda_2$, does there exist a Lagrangian concordance $\Lambda_2\prec \Lambda_1$? In this case, the Thurston-Bennequin number is not an obstacle.
The general answer is no. For instance:
\begin{ex}\label{lem:no_conc}(see \cite{Chantraine2015sym})
There exists a Lagrangian concordance between a max-tb unknot and the max-tb $m(9_{46})$ knot. There does not exist a Lagrangian concordance in the opposite direction.
\end{ex}

\begin{proof}
The concordance $\Upsilon\prec m(9_{46})$ is depicted in Figure \ref{fig:946_lin}. The reverse concordance can be obstructed in a number of ways, including by counts of normal rulings or augmentations in the LCH (Legendrian contact homology) of the $m(9_{46})$ knot (see \cite{CornwellNgSivek2016}, as well as \cite{EHK},\cite{pan_augm} for descriptions of the invariants).
\end{proof}

All known obstructions vanish after adding stabilizations (which trivialize the invariants). 
Indeed, recent results (cf \cite{Gol1},\cite{Gol2}) show that given a decomposable concordance $\Lambda_1\prec \Lambda_2$, there must exist a concordance $S(\Lambda_2)\prec S(\Lambda_1)$ after adding enough stabilizations.
We propose a technique that can be used to explicitly reverse decomposable concordances using Lagrangian moves:
\begin{itemize}
\item replace $1$-handle attachments of type $E_1$/$H_1$ with $H_2$ handles;
\item replace any $0$-handle with a Lagrangian `unknot trick' $U_1$ (cf Fig \ref{fig:curls}).
\end{itemize}

As mentioned, this blueprint can't generally work without also adding stabilizations (more on this in Section \ref{sec:areas}).
Let us first see how it works in a toy example, simple enough to not require stabilizations.

\subsection{Inverting a concordance: toy example}\label{sec:toy_example}
\begin{figure}[h]
    \begin{subfigure}{0.25\textwidth}
        \centering

        \includegraphics[width=\linewidth]{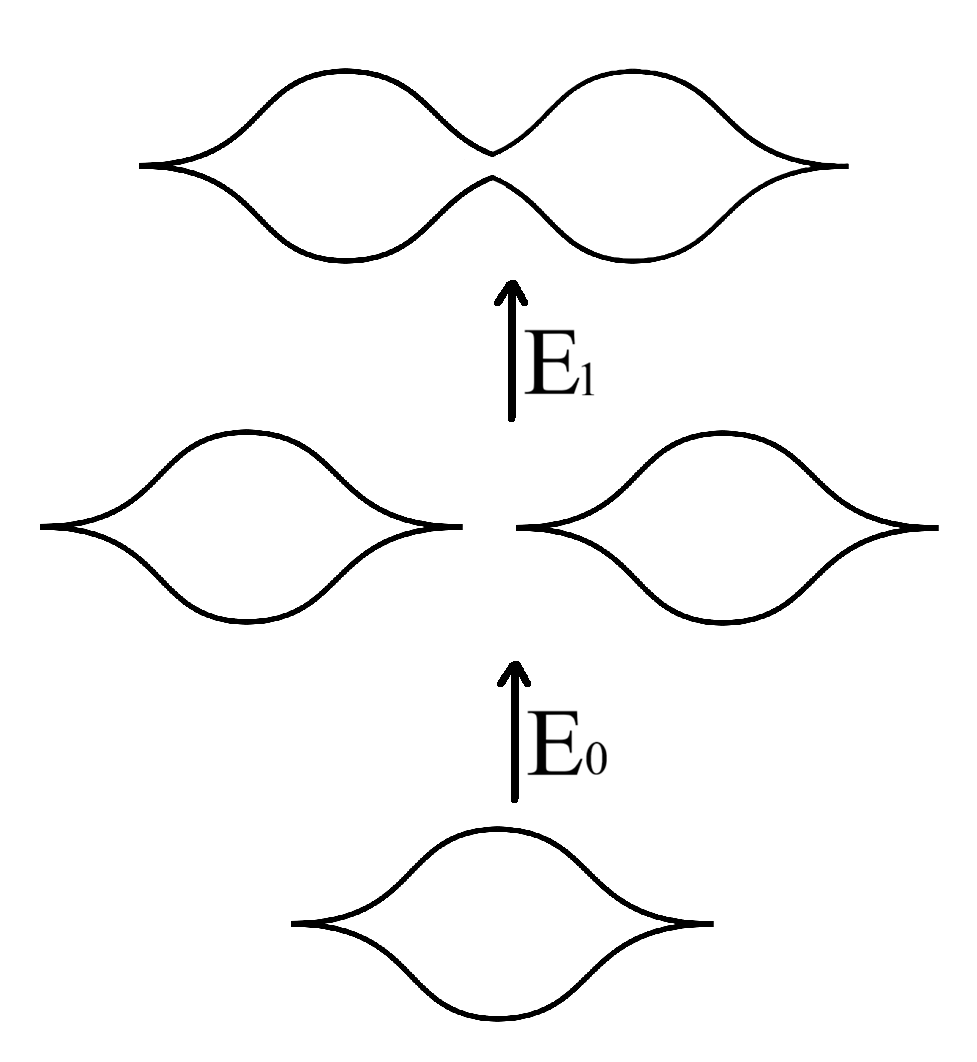}
        \caption{Decomposable concordance}
        \label{fig:toy_leg}
    \end{subfigure}
    \hfill
    \begin{subfigure}{0.3\textwidth}
        \centering
        \includegraphics[width=\linewidth]{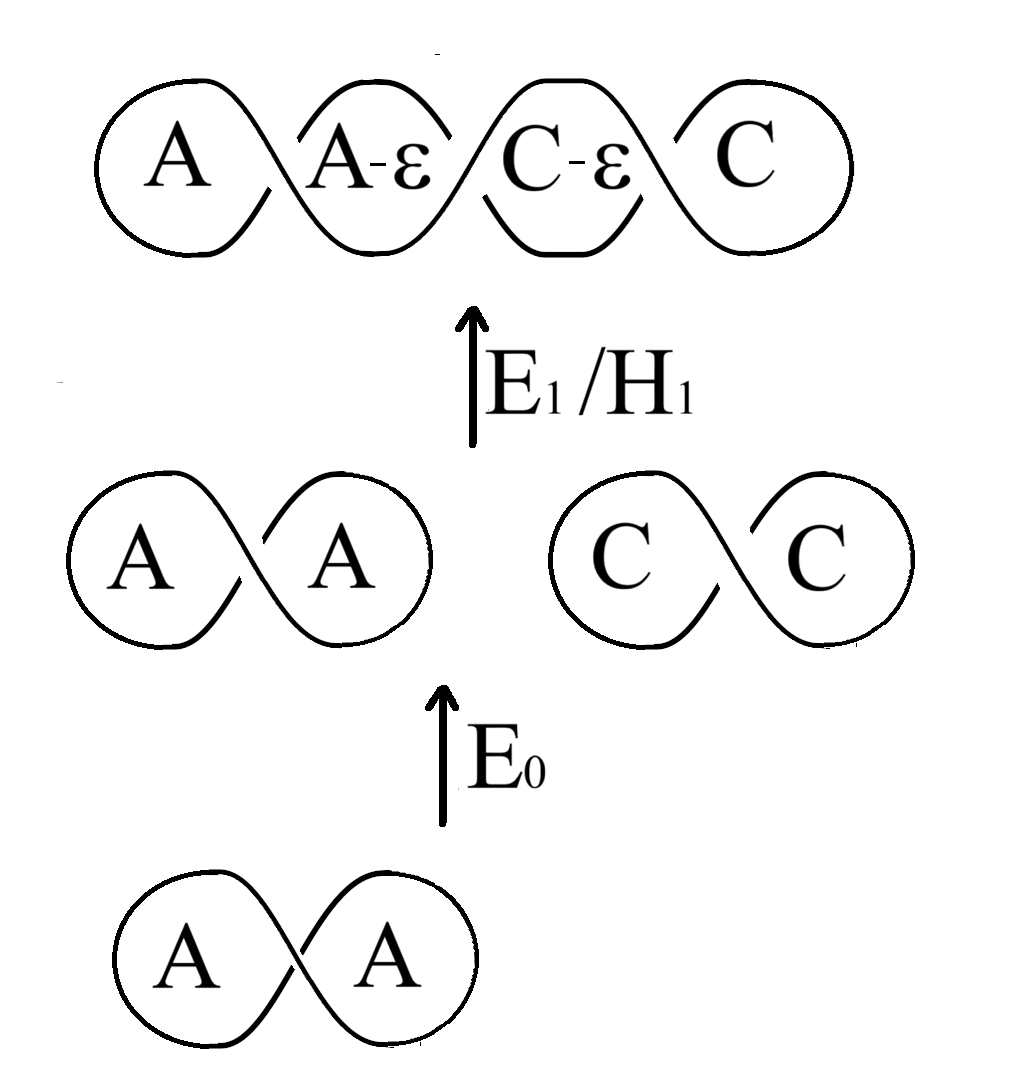}
        \caption{Lagrangian interpretation}
        \label{fig:toy_lag}
    \end{subfigure}
    \hfill
    \begin{subfigure}{0.35\textwidth}
        \centering
        \includegraphics[width=\linewidth]{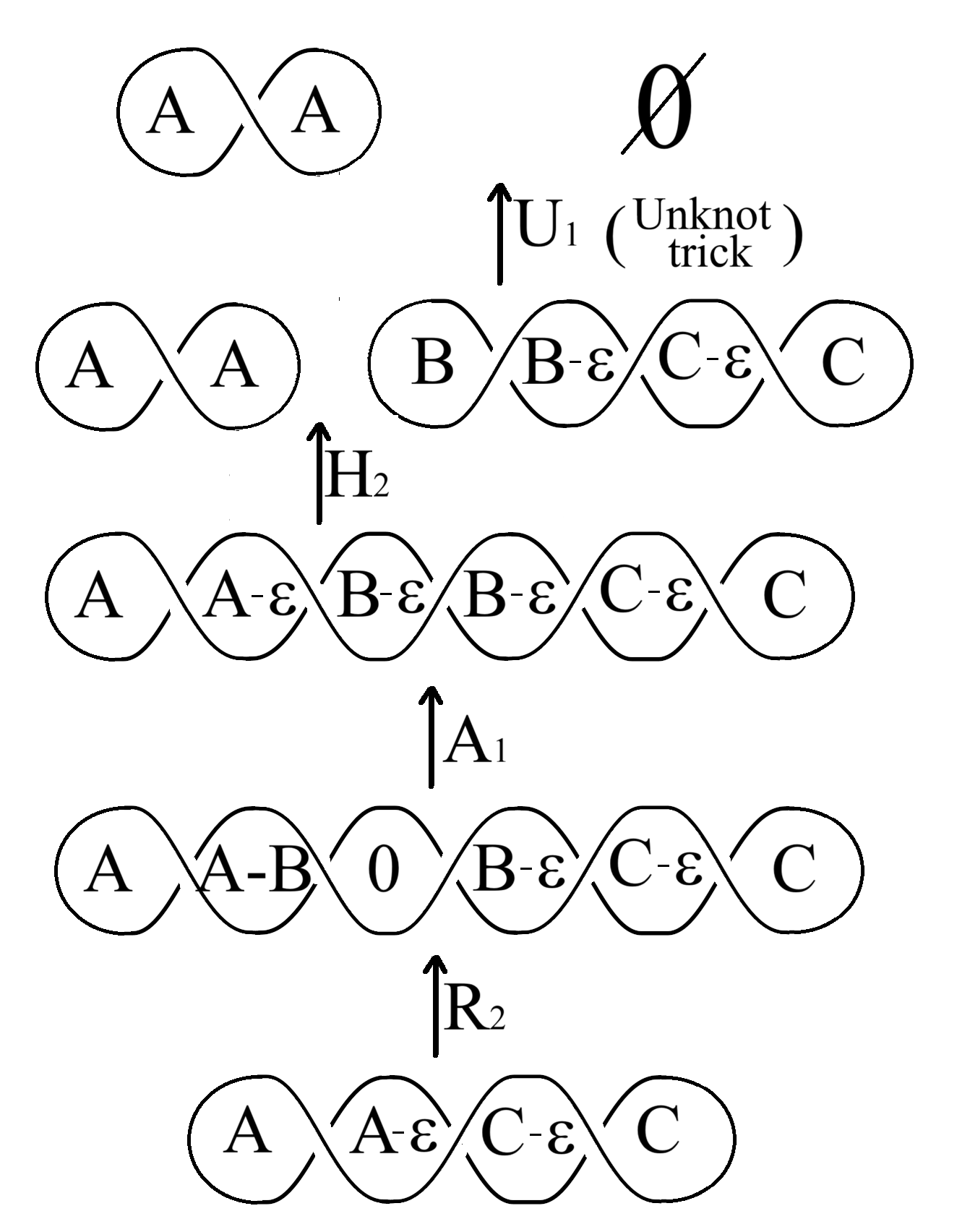}
        \caption{Inverse concordance}
        \label{fig:toy_sol}
    \end{subfigure}
    \caption{Toy example of inversion.}
    \label{fig:easy_invert}
\end{figure}
The technique is illustrated in Figure \ref{fig:easy_invert}. Start with a decomposable concordance of unknots: a $0$-handle attachment followed by a $1$-handle attachment (no isotopies) as in \ref{fig:toy_leg}, \ref{fig:toy_lag}. To build an inverse concordance, start at the top diagram of \ref{fig:toy_lag}. Assume $A>C$. Perform an $R_2$ move followed by $A_1$, to create the areas marked $B$ in the figure, ensuring that $B>C$.
`Invert' the $1$-handle attachment by performing $H_2$.
In this simple example, 
the two resulting components are separated in the $\R^2$ projection; what was previously a $0$-handle attachment can be replaced by a $U_1$ cap, and we are done.

\subsection{Moving areas (trickle tricks)}\label{sec:areas}
It is not always possible to directly apply the method from the previous section and invert a concordance. We have already discussed the obstructions from an LCH perspective: these can be circumvented by adding stabilizations. 

From the perspective of Lagrangian moves, the obstruction presents in the form of areas: larger area regions cannot freely move through smaller area regions (a problem reminiscent of the symplectic camel). 
The technique requires generating a sufficiently large $B$ area, then splitting the components and separating them in the plane while maintaining $B>C$ ($U_1$ only applies if $B>C$).

\begin{figure}
    \begin{subfigure}{0.45\textwidth}
        \centering
        \includegraphics[width=\linewidth]{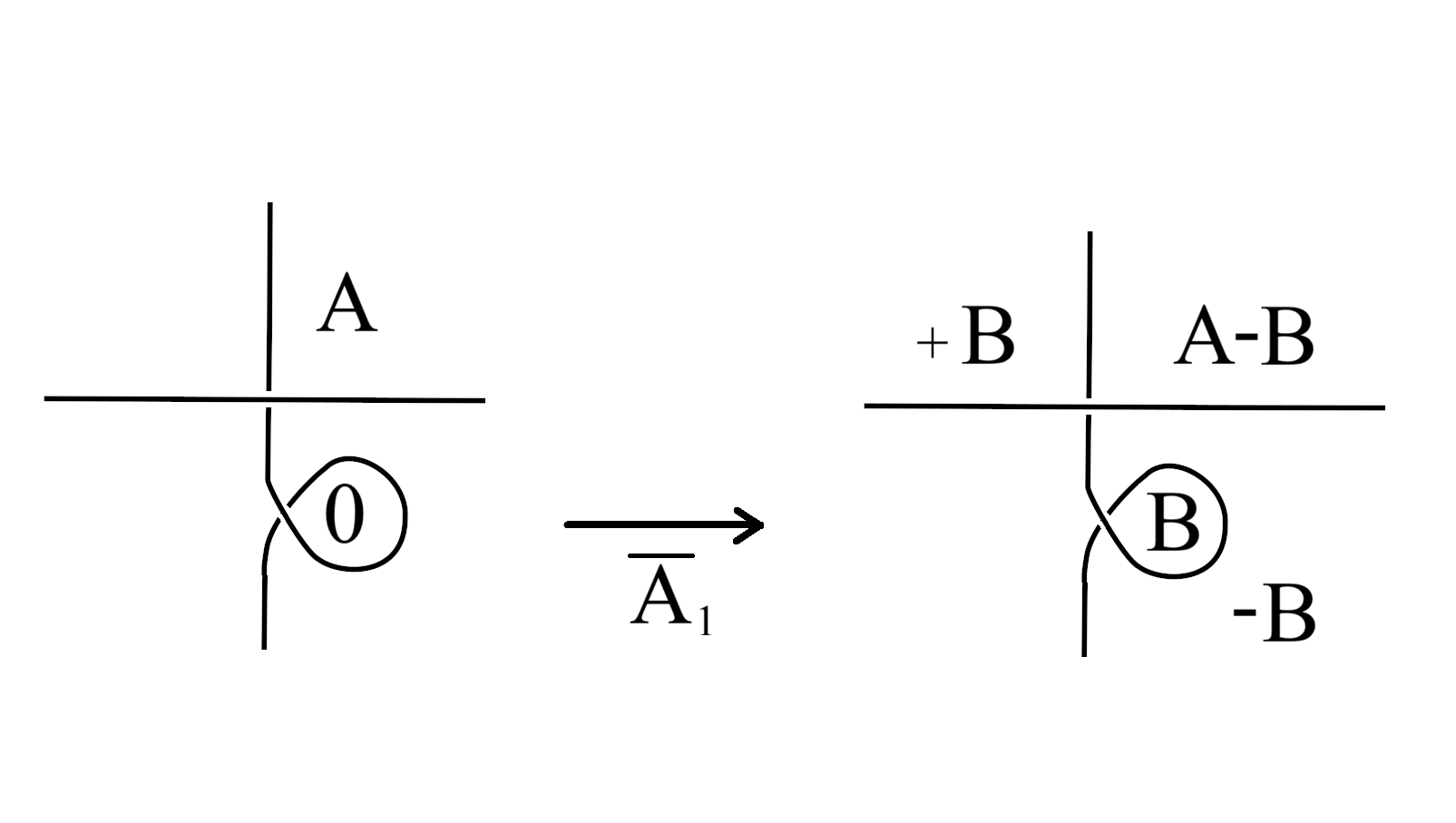}
        \caption{$\bar{A_1}$ ($A\geq 2B$)}
        \label{fig:invert_A1}
    \end{subfigure}
    \hfill
    \begin{subfigure}{0.45\textwidth}
        \centering
        \includegraphics[width=\linewidth]{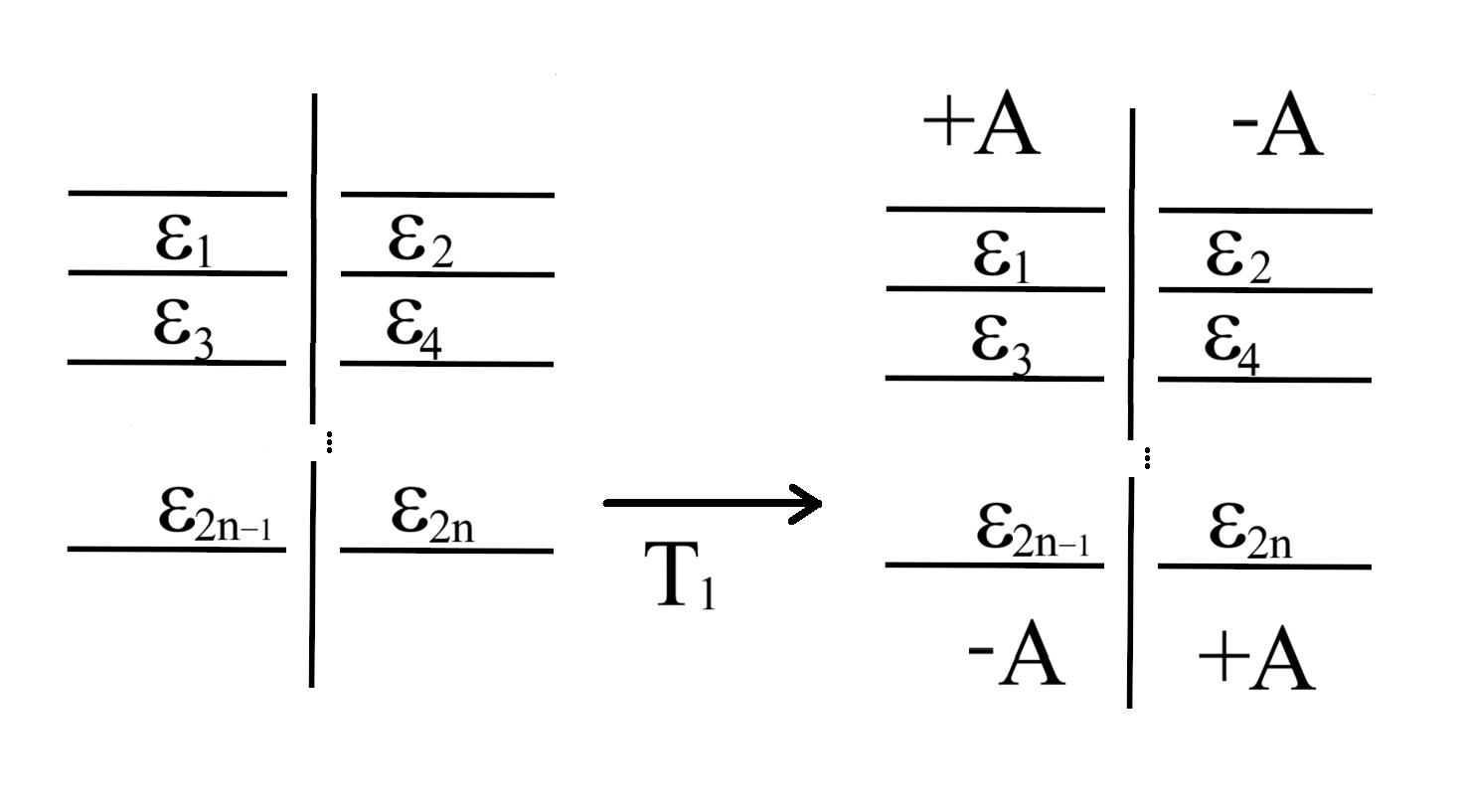}
        \caption{Trickle $T_1$}
        \label{fig:trickle1}
    \end{subfigure}

    \vspace{0.3cm}
    \begin{subfigure}{0.45\textwidth}
        \centering
        \includegraphics[width=\linewidth]{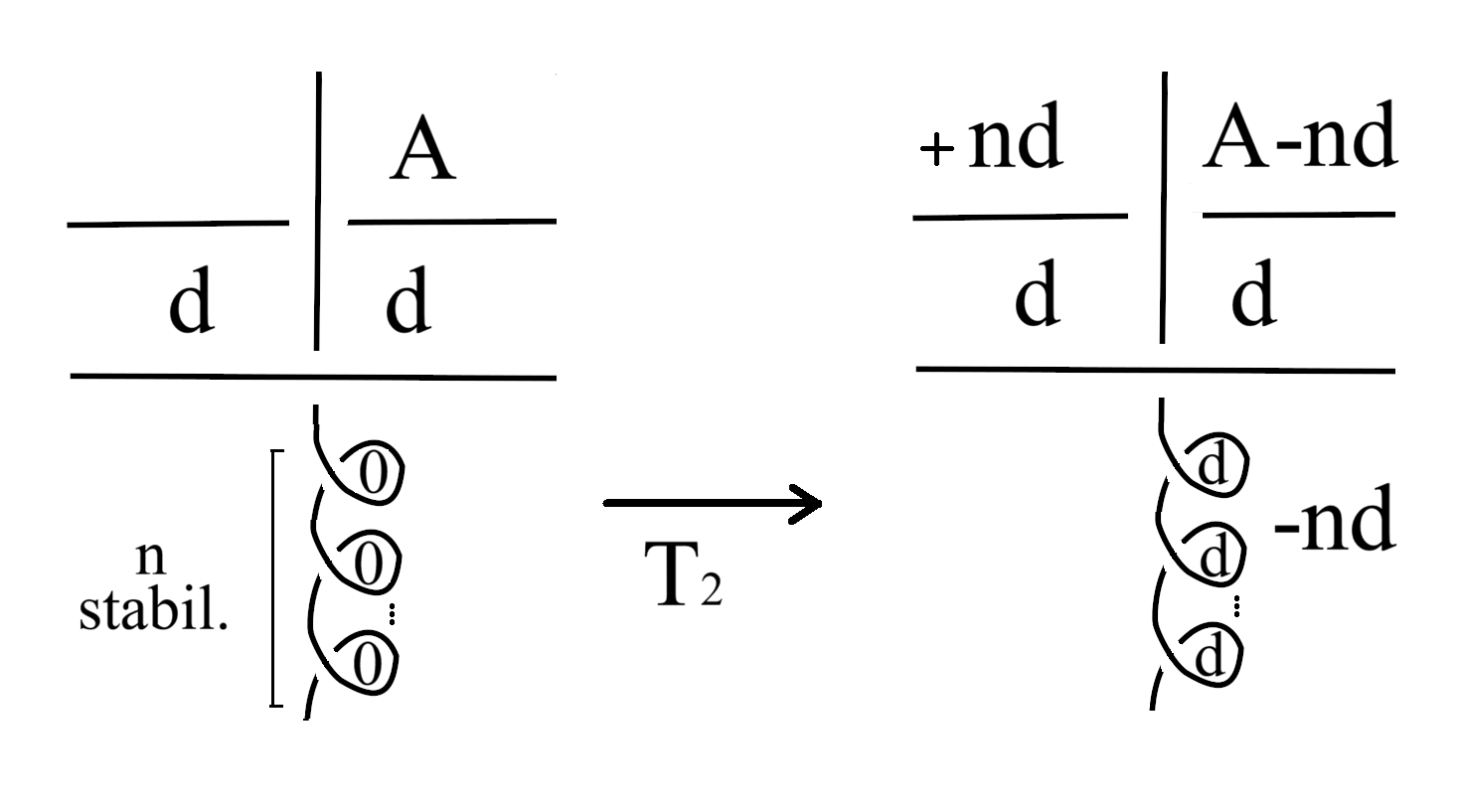}
        \caption{Trickle $T_2$}
        \label{fig:trickle2}
    \end{subfigure}
    \hfill
    \begin{subfigure}{0.45\textwidth}
        \centering
        \includegraphics[width=\linewidth]{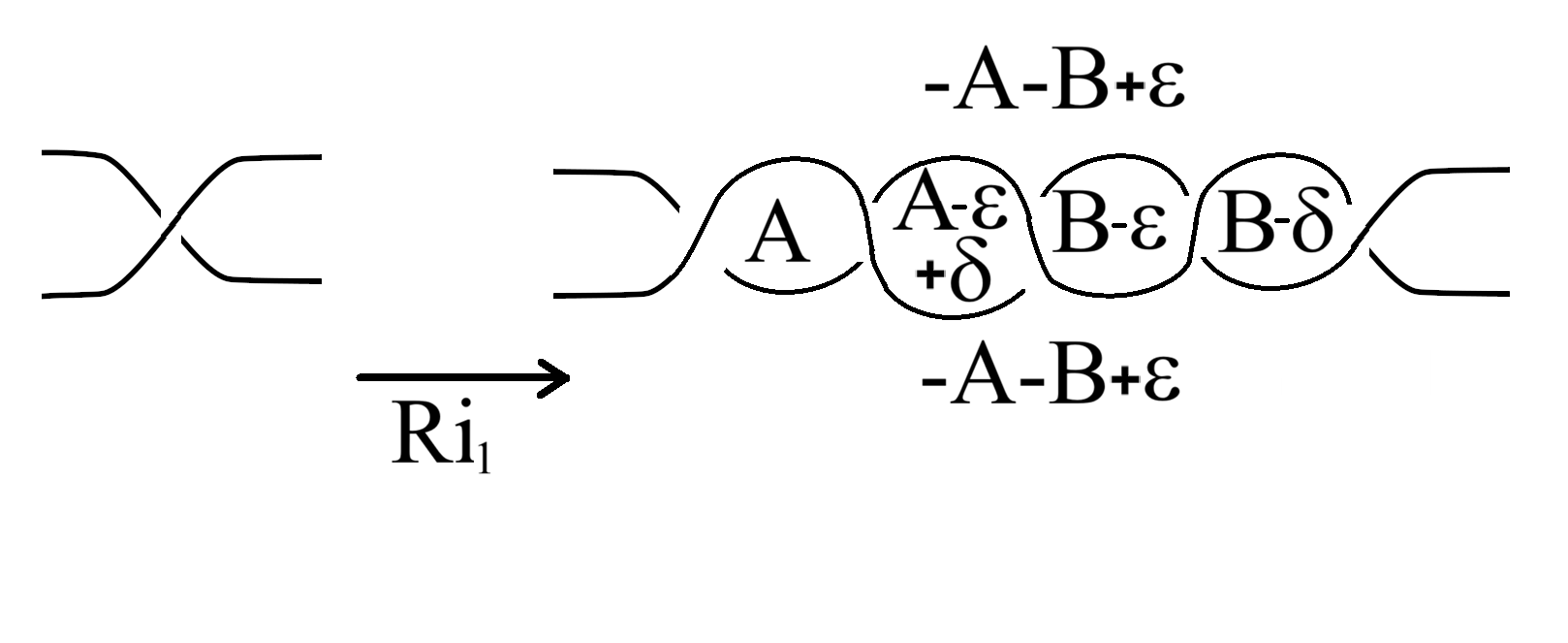}
        \caption{Ribbon move: $A>B\geq\varepsilon\geq0;\\ B-A<\delta<B$}
        \label{fig:ri1}
    \end{subfigure}

    \caption{Trickle moves and ribbon trick.}
    \label{fig:trickles}
\end{figure}
Stabilizations provide the necessary flexibility for this. To see why, observe that the $A_1$ move is directional and can't be reversed. However, a stabilization can be used to partially undo an $A_1$ move, obtaining $\bar{A_1}$ (see Figure \ref{fig:invert_A1}). We rely on $\bar{A_1}$ to build $T_2$. The goal of $T_1,T_2$ is to displace large areas in the diagram. Moves $\bar{A_1},T_1,T_2$ also have analogs when crossing signs are swapped.

\begin{lem}\label{lem:trickles} 
The moves in Figure \ref{fig:trickles} are built of Lagrangian moves $R_2,R_3,A_1$.
\end{lem}





\begin{figure}[!h]
    \begin{subfigure}{0.45\textwidth}
        \centering
        \includegraphics[width=\linewidth]{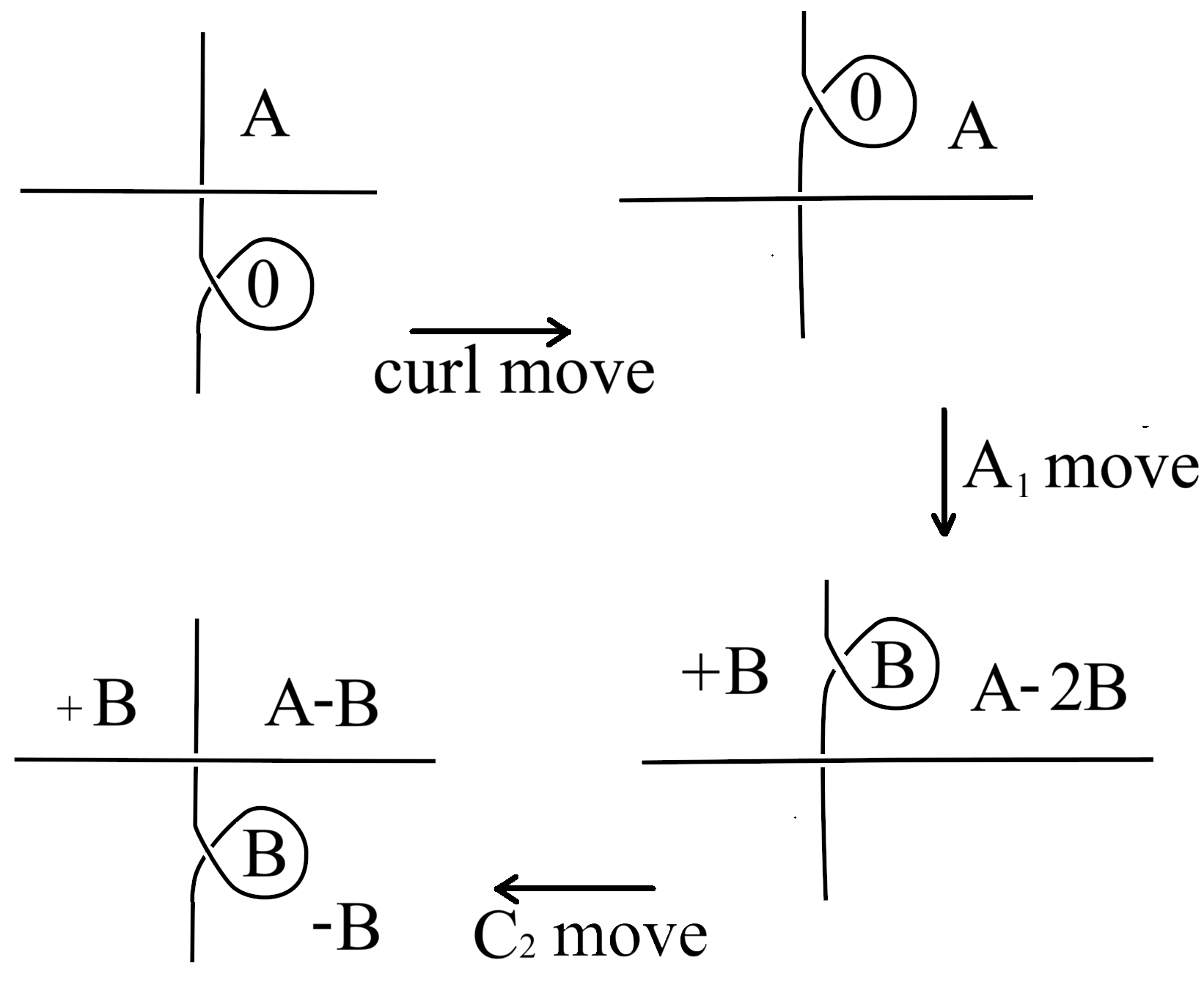}
        \caption{Proof of $\bar{A_1}$ ($A\geq 2B$)}
        \label{fig:invertA1_proof}
    \end{subfigure}
    \hfill
    \begin{subfigure}{0.45\textwidth}
        \centering
        \includegraphics[width=\linewidth]{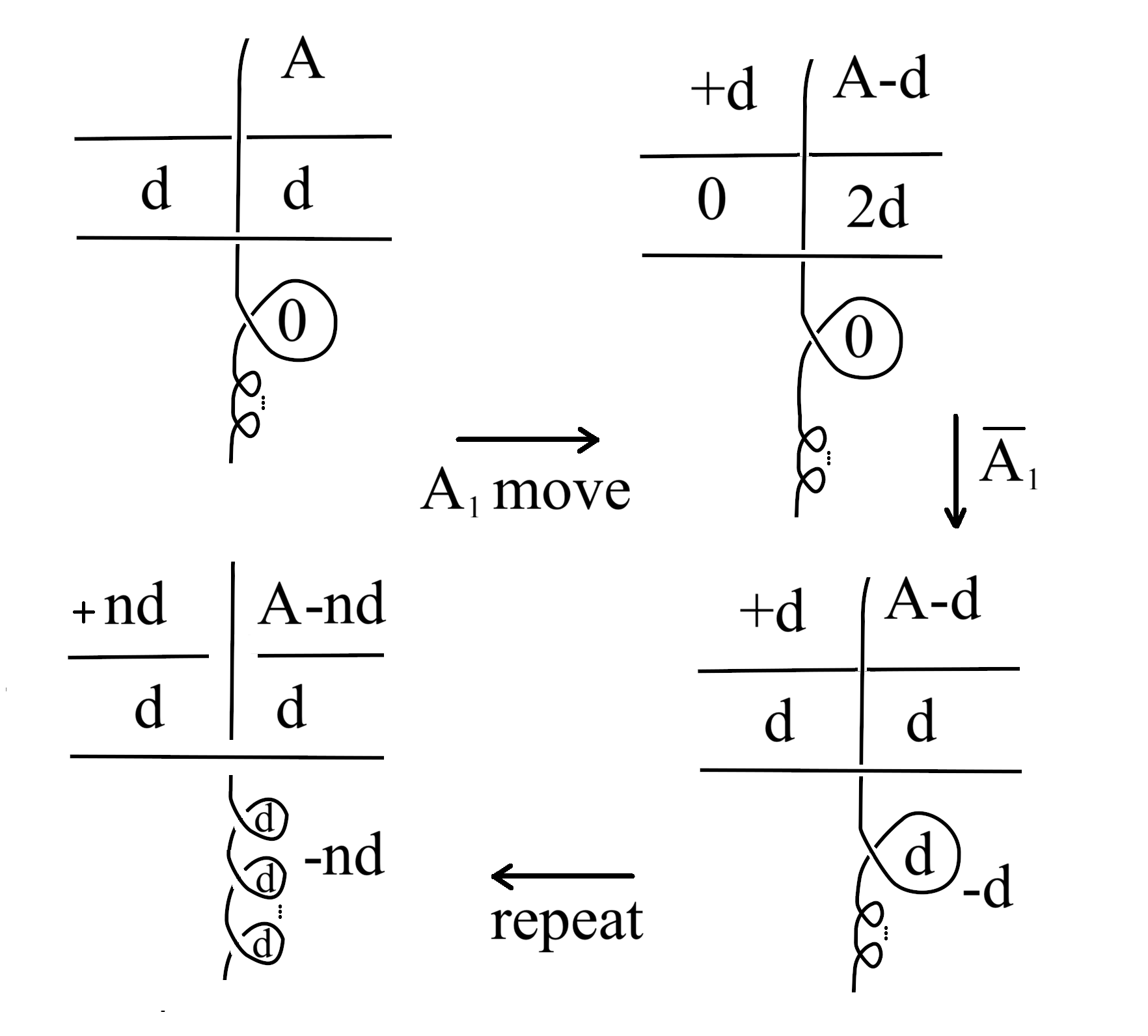}
        \caption{Proof of $T_2$}
        \label{fig:trickle2_proof}
    \end{subfigure}

    \caption{Moving area using stabilizations.}
    \label{fig:trickle_proofs}
\end{figure}

\begin{proof}
For the trickle move $T_1$: perform $n+1$ consecutive $A_1$ moves of size $\varepsilon<\varepsilon_{1,\ldots,2n}$ to displace an $\varepsilon$ area; repeat the process $\frac{A}{\varepsilon}$ times.
For the ribbon move $Ri_1$: perform $R_2$ followed by an area move of amplitude $A$ to create two $A$ areas; then we proceed similarly to the toy example: perform a second $R_2$ move separating $A-B+\delta$ area on one side, $B-\delta$ on the other; an additional area move of amplitude $B-\varepsilon$ increases the central areas to $A-\varepsilon+\delta$, $B-\varepsilon$.

The proofs for moves $\bar{A_1}$, $T_2$ are shown in Figure \ref{fig:trickle_proofs} (they make use of curl moves, which are built from $R_2,R_3,A_1$).
\end{proof}

\subsection{Reversing concordances: general blueprint}\label{sec:bluep}

    \begin{figure}[h]
        \centering
        \includegraphics[width=.6\linewidth]{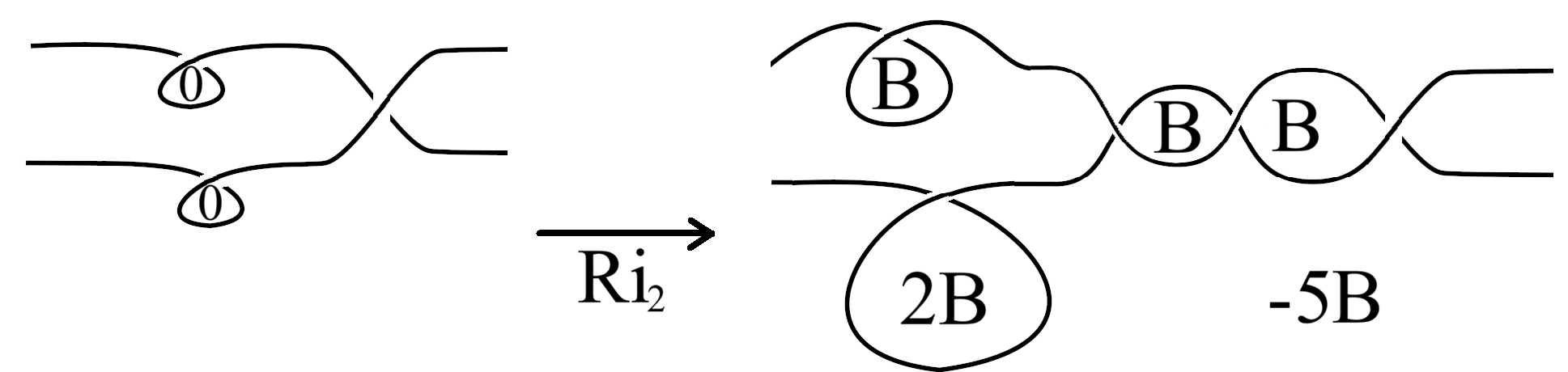}
    \caption{One-sided ribbon trick}
	    \label{fig:ribbon2}
    \end{figure}

For clarity, we restrict to the case of a concordance obtained by a single $0$-handle/$1$-handle pair.
We assume a $0$-handle attachment is followed by the addition of a ribbon joining the $0$-handle to the rest of the knot (the ribbon may be knotted around itself, the original knot, and the $0$-handle). The decomposable $1$-handle is then located along the ribbon. 

We propose the following blueprint for reversing the concordance:
\begin{itemize}
\item Start at the top end; via isotopy, ensure that, close to the $H_1$ move, at least one side of the ribbon is bordering a large area;
\item Perform ribbon trick $Ri_1$ (or the one-sided alternative, $Ri_2$ in Figure \ref{fig:ribbon2}, obtained with two additional stabilizations and some area moves);
\item Detach using the $H_2$ move: this results in two smoothly unlinked, but symplectically linked, knots;
\item Unwind the ribbon to unlink the two components; the unwinding may require moves $T_{1,2}$; $T_2$ requires adding stabilizations;
\item Once the components are unlinked, apply $U_1$.
\end{itemize}

This technique is successful in many cases; we will check the simplest one.

\begin{figure}[!h]
    \begin{subfigure}{\textwidth}
        \centering

        \includegraphics[width=0.7\linewidth]{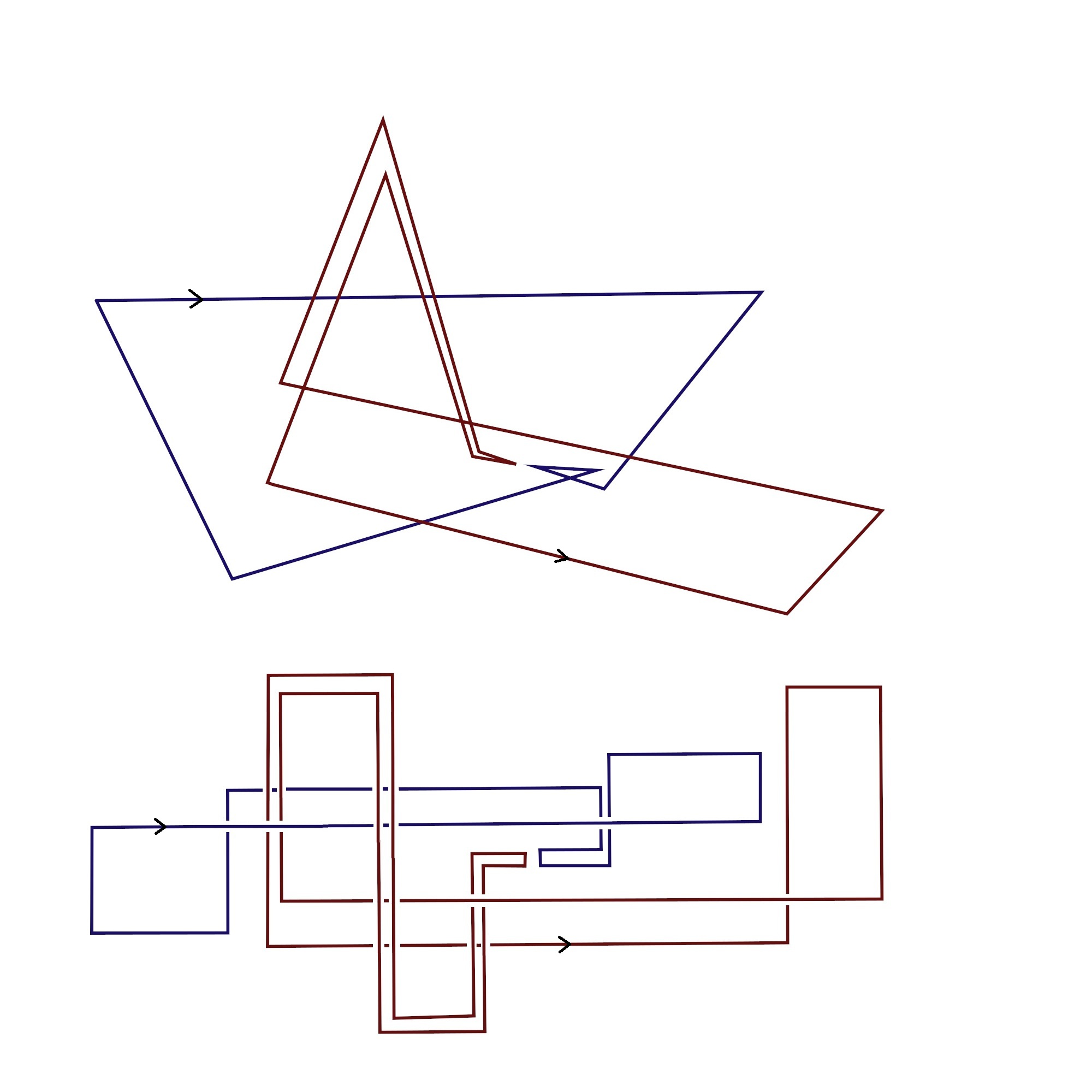}
        \caption{Linear front and Lagrangian projection before attachment (two unknots).}
        \label{fig:946_detached}
    \end{subfigure}

    \begin{subfigure}{\textwidth}
        \centering

        \includegraphics[width=0.6\linewidth]{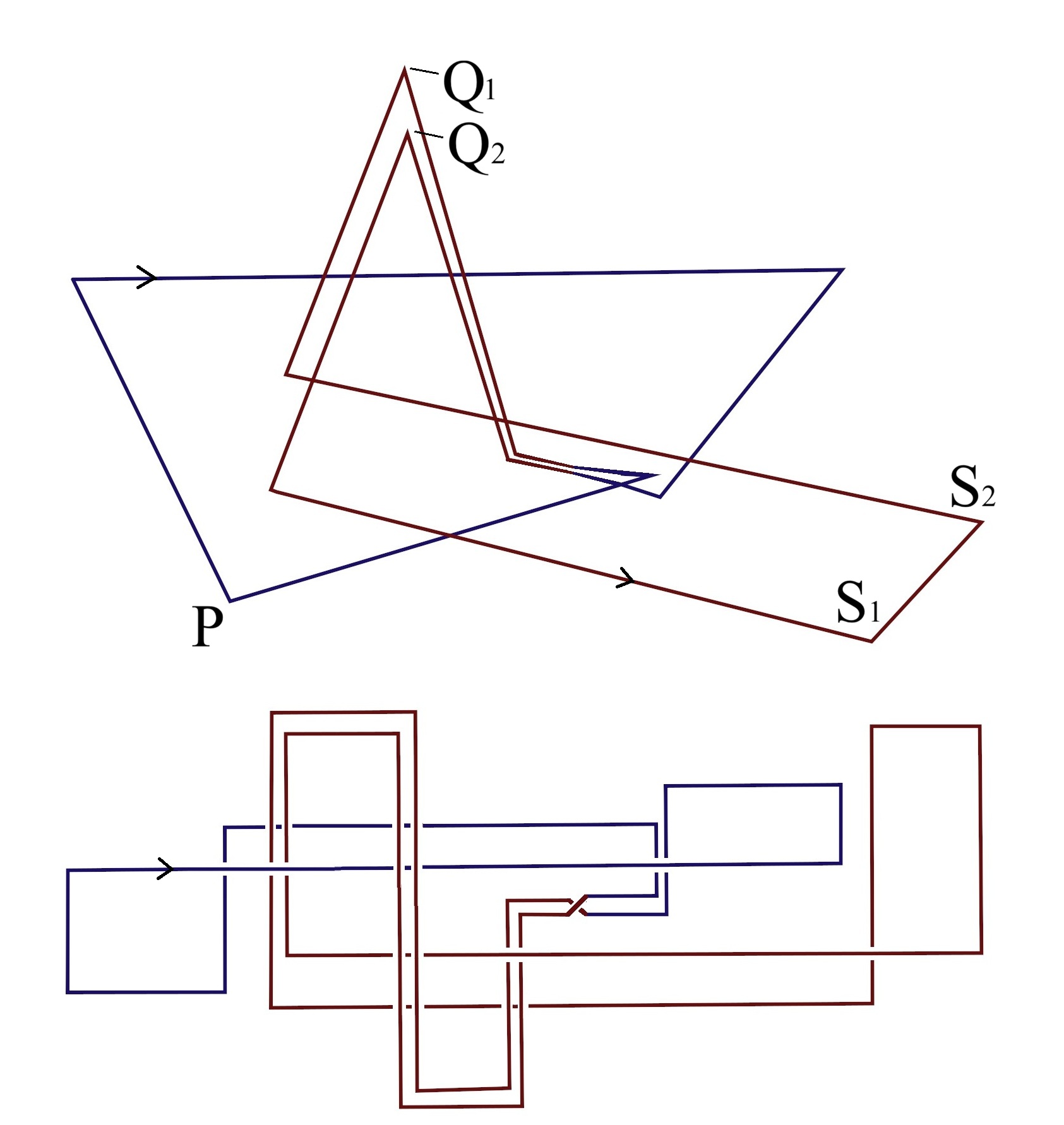}
        \caption{Linear front and Lagrangian projection of $m(9_{46}).$}
        \label{fig:946}
    \end{subfigure}
    
    \caption{Linear front diagrams and Lagrangian projections for $\mathcal{C}$.}
    \label{fig:946_lin}
\end{figure}
\section{The main example}\label{sec:main_example}


\begin{figure}[!h]
    \centering
    \makebox[\textwidth][c]{%
        \includegraphics[width=.8\linewidth]{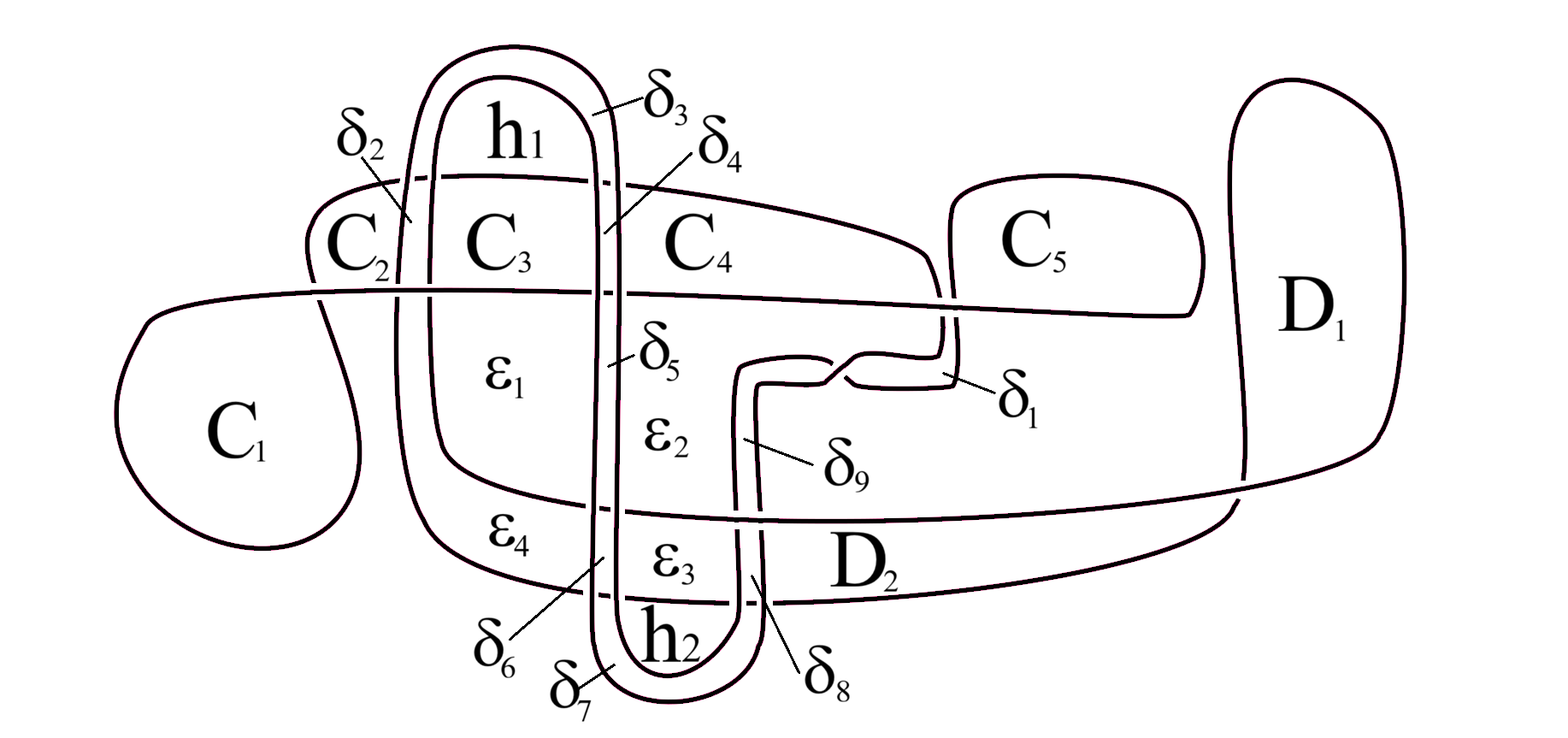}
    }
    \caption{A Lagrangian diagram for the Legendrian $m(9_{46})$ knot}
    \label{fig:step0}
\end{figure}

The goal of this section is Theorem \ref{thm:main_example}, which is proven by reversing the decomposable concordance $\mathcal{C}:\Upsilon\prec m(9_{46})$. We mainly follow the steps from Section \ref{sec:bluep}, with some additional ad hoc diagram manipulation to reduce the number of stabilizations needed, yielding $C:S_+^{4}S_-(m(9_{46}))\prec S_+^{4}S_-(\Upsilon)$.
\subsection{Diagram setup}
Figure \ref{fig:946_lin} shows the linear front diagram and Lagrangian projection of $\mathcal{C}$.
Specifically, Figure \ref{fig:946_detached} shows a linear front and corresponding Lagrangian diagram for two unlinked unknots (right before the handle attachment producing the $m(9_{46})$ knot). Figure \ref{fig:946} shows the same linear front and corresponding Lagrangian diagram for the $m(9_{46})$ knot (after the handle attachment).

Thanks to the linearized diagrams we obtain:
\begin{lem}\label{lem:step0}
The diagram in Figure \ref{fig:step0} represents a max-tb, Legendrian realization of the $m(9_{46})$ knot. The areas satisfy $$C_1+\delta_1 = C_2+C_3+C_4+C_5+\delta_2+\delta_4+\bar{\delta},$$
$$D_1=D_2+\varepsilon_3+\varepsilon_4+\delta_2+\delta_3+\ldots+\delta_9 +\bar{\delta}$$
where $\bar{\delta}$ is an infinitesimal correction term. Also, $\delta_1,\ldots,\delta_9$ can be assumed to be much smaller than all other areas.
\end{lem}

\begin{rmk}
For clarity, we avoid explicitly computing all areas in Figure \ref{fig:946_lin}.
\end{rmk}

\begin{figure}[htbp]
    \centering
    \makebox[\textwidth][c]{
        \includegraphics[width=1.3\linewidth]{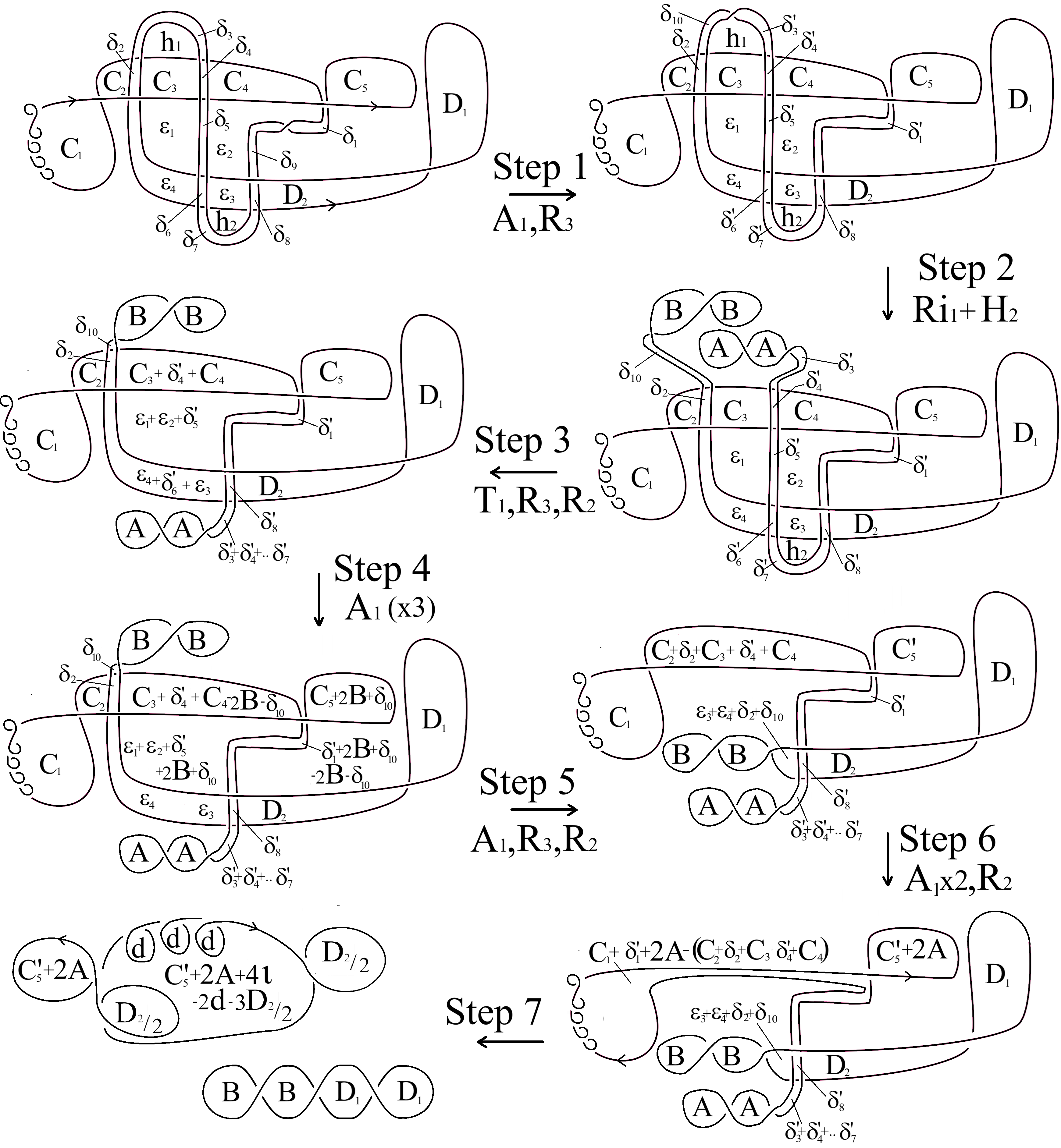}
}
    \caption{Lagrangian moves from a stabilized $m(9_{46})$ to a stabilized unknot.}
    \label{fig:proof}
\end{figure}

\begin{figure}[!h]
    \centering
    \makebox[\textwidth][c]{%
        \includegraphics[width=1.3\linewidth]{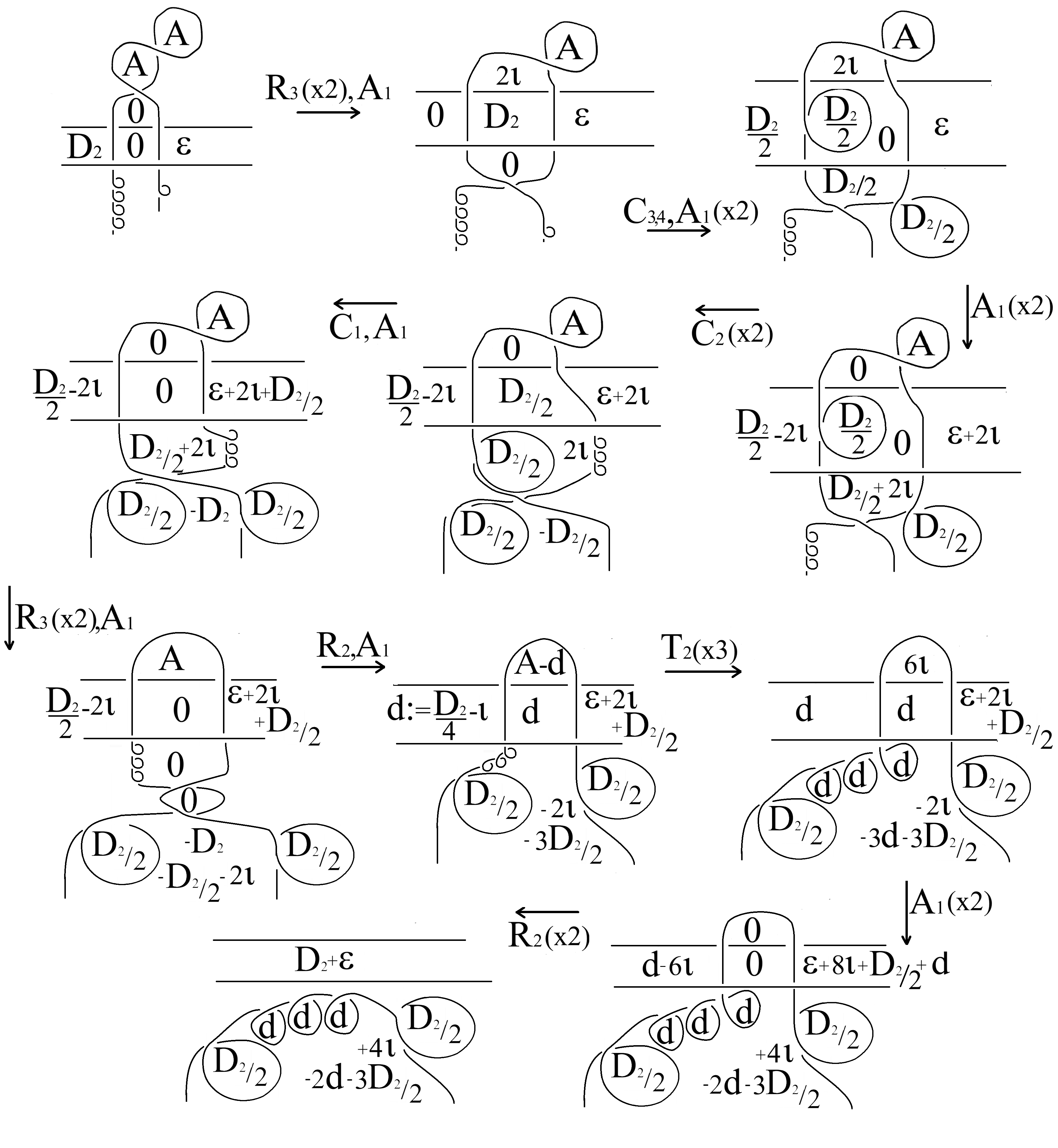}
}
    \caption{Details for Step 7.}
    \label{fig:step7}
\end{figure}

\begin{proof}[Proof of Lemma \ref{lem:step0}]
Figure \ref{fig:946_lin} shows the derivation of the Lagrangian diagram. Compare Figure \ref{fig:946} to \ref{fig:946_detached}. In \ref{fig:946_detached}, each component is a Legendrian unknot, thus the corresponding areas satisfy $C_1+\delta_1 = C_2+C_3+C_4+C_5+\delta_2+\delta_4$ and $D_1=D_2+\varepsilon_3+\varepsilon_4+\delta_2+\delta_3+\ldots+\delta_9$. In  \ref{fig:946} the same is true up to a correction term due to the handle attachment (see move $H_1$ in Figure \ref{fig:lagr_moves}).

To see that $\delta_1,\ldots,\delta_9$ can be made arbitrarily small, notice that the ribbon they create in the Lagrangian diagram corresponds to a ribbon in the (linear) front projection (refer again to Figure \ref{fig:946_lin}). The ribbon in the linear front can be made arbitrarily thin (up to isotopy) and with almost-parallel sides. This, in turn, will make the ribbon in the Lagrangian projection arbitrarily thin.
\end{proof}

\subsection{Constructing the concordance}

Starting from the diagram in Figure \ref{fig:step0}, one could apply the ribbon trick $Ri_2$ and then proceed with detaching and unwinding the resulting components. We add a few steps to be able to use $Ri_1$ instead, which uses fewer stabilizations.


\begin{prop}\label{prop:main_diagrams}
There is a sequence of Lagrangian moves starting from the Lagrangian diagram associated to $S_+^{4}S_-(m(9_{46}))$ and ending at the diagram associated to $S_+^{4}S_-(\Upsilon)$. The sequence includes $H_2$ and $U_1$ exactly once, and further only $R_2,R_3,A_1$.
\end{prop}


\begin{proof}
Start with a Lagrangian diagram of $m(9_{46})$ as in Figure \ref{fig:946}. We can tweak the diagram via Legendrian isotopy to ensure a few properties for the resulting diagram in Figure \ref{fig:step0}. In the following order:
\begin{itemize}
\item move $S_1,S_2$ to the right, increasing $D_{1,2}$, until $0<\iota:=D_1-D_2<\frac{D_1}{100}$;
\item lower $P$ and raise $Q_1,Q_2$ along $z$, until $C_2,C_3,C_4>2(D_1+\iota)$;
\item further increase the $z$-coordinate of $Q_1,Q_2$ to ensure $h_1,h_2>2(D_1+\iota)$;
\item resize the ribbon to ensure $\delta_1,\ldots\delta_9\ll\varepsilon_1,\ldots,\varepsilon_4,\iota$.
\end{itemize}

We can then add the needed stabilizations. From this point, the proof follows Figure \ref{fig:proof}. Step 1 is straightforward, given that all $\delta,\delta'$ are very small. For Step 2, let $B=D_1+\frac{\iota}{2}, A=D_1+\iota$; the ribbon trick can be applied since $h_1>A+B$. We moreover can choose $\delta$ in the ribbon trick to match the error term $\bar{\delta}$ from Lemma \ref{lem:step0}, so that, after $H_2$, the two resulting link components both have total area $0$. For readability, such $\delta$ is ignored in the figure.

Step 3 is analogous to Step 1, with the addition of $T_1$ to displace the larger $A$ areas. Step 4 enlarges the area $\varepsilon_1+\varepsilon_2+\delta'_5$ in preparation of Step 5. Step 5 is then possible since $C_2>2B+\delta_{10}$ and $\varepsilon_1+\varepsilon_2+\delta'_5 +2B>2B+\delta_{10}$. Step 6 is a straightforward diagram simplification.

Step 7 is the core of the proof: this is where the stabilizations are needed; there are multiple ways to solve this step, applying $T_2$ repeatedly to move all the $A$ areas through the $D_2$ area. Our proposed solution is locally expanded in Figure \ref{fig:step7}. Here, infinitesimal areas are denoted by $0$ for clarity as they don't affect the proof. We make repeated use of the equality $A=D_2+2\iota$. We also let $d=\frac{D_2}{4}-\iota$ for readability.

Step 7 produces two separate components, both with total area $0$ thanks to Step 2. One of these can be capped via a $U_1$ move, while the second one is a stabilized unknot (this is easy to see directly, since we reduced it to a standard form of a large eight-shape with five additional, smaller loops). 

To conclude the proof, notice that the only handles are $H_2$ and $U_1$.
\end{proof}

\begin{proof}[Proof of Theorem \ref{thm:main_example}]
Proposition \ref{prop:main_diagrams} provides a sequence of Lagrangian moves from $D_{S_+^{4}S_-(m(9_{46}))}$ to $D_{S_+^{4}S_-(\Upsilon)}$. The sequence starts as a connected knot, adds a single $1$-handle attachment $H_2$ (in Step 2) splitting it into two components, followed by a single cap of one component (at the very end). So, the cobordism obtained from Proposition \ref{thm:leg_extension} is smoothly a concordance.
Corollary \ref{cor:lagr_conc} then ensures it is a Lagrangian concordance of Legendrians.
\end{proof}

\bibliography{mybib}

@article{EHK,
    author = {Ekholm, Tobias and Honda, Ko and K{\'a}lm{\'a}n, Tam{\'a}s},
    title = {Legendrian knots and exact {L}agrangian cobordisms},
    journal = {Journal of the European Mathematical Society},
    volume = {18},
    year = {2016},
    number = {11},
    pages = {2627--2689},
    doi = {10.4171/JEMS/650}
}

@article{lin,
  author    = {Lin, Francesco},
  title     = {Exact {L}agrangian caps of {L}egendrian knots},
  journal   = {Journal of Symplectic Geometry},
  year      = {2016},
  volume    = {14},
  number    = {1},
  pages     = {269-295},
  doi       = {10.4310/JSG.2016.v14.n1.a10},
  publisher = {International Press of Boston}
}

@article{Gol1,
  author  = {Dimitroglou Rizell, Georgios and Golovko, Roman},
  title   = {Instability of Legendrian knottedness, and non-regular Lagrangian concordances of knots},
  journal = {Advances in Mathematics},
  volume  = {502},
  year    = {2026},
  eid     = {111133},
  doi     = {10.1016/j.aim.2026.111133}
}

@article{Gol2,
  author  = {Dimitroglou Rizell, Georgios and Golovko, Roman},
  title         = {Non-regular Lagrangian concordances between Lagrangian fillable Legendrian knots},
  journal       = {arXiv preprint},
  eprint        = {2509.13594},
  archivePrefix = {arXiv},
  year          = {2025},
  note          = {To appear in the Proceedings of the American Mathematical Society}
}

@article{pan_augm,
  author    = {Pan, Yu},
  title     = {The augmentation category map induced by exact {L}agrangian cobordisms},
  journal   = {Algebraic \& Geometric Topology},
  volume    = {17},
  number    = {3},
  pages     = {1813--1870},
  year      = {2017},
  publisher = {Mathematical Sciences Publishers}
}

@article{CornwellNgSivek2016,
  author    = {Cornwell, Christopher R. and Ng, Lenhard and Sivek, Steven},
  title     = {Obstructions to {L}agrangian concordance},
  journal   = {Algebraic \& Geometric Topology},
  year      = {2016},
  volume    = {16},
  number    = {2},
  pages     = {797--824},
  doi       = {10.2140/agt.2016.16.797},
  publisher = {Mathematical Sciences Publishers}
}

@article{chantraine2010concordance,
  author  = {Chantraine, Baptiste},
  title   = {Lagrangian concordance of {L}egendrian knots},
  journal = {Algebraic \& Geometric Topology},
  volume  = {10},
  number  = {1},
  pages   = {63--85},
  year    = {2010},
  doi     = {10.2140/agt.2010.10.63}
}

@article{Datta,
  author    = {Datta, Ipsita},
  title     = {Lagrangian Cobordisms between Enriched Knot Diagrams},
  journal   = {Journal of Symplectic Geometry},
  volume    = {21},
  number    = {1},
  pages     = {159--234},
  year      = {2023},
  publisher = {International Press},
  doi       = {10.4310/JSG.2023.v21.n1.a4},
  eprint    = {2112.10015},
  archivePrefix = {arXiv}
}

@incollection{survey,
  author    = {Blackwell, Sarah and Legout, No{\'e}mie and Leverson, Caitlin and Limouzineau, Ma{\"y}lis and Myer, Ziva and Pan, Yu and Pezzimenti, Samantha and Su{\'a}rez, Lara Simone and Traynor, Lisa},
  title     = {Constructions of {L}agrangian cobordisms},
  booktitle = {Research Directions in Symplectic and Contact Geometry and Topology},
  series    = {Association for Women in Mathematics Series},
  volume    = {27},
  pages     = {245--272},
  publisher = {Springer},
  address   = {Cham},
  year      = {2021},
  primaryClass  = {math.SG},
  doi       = {10.1007/978-3-030-80979-9_5},
  url       = {https://arxiv.org/abs/2101.00031}
}

@article{ELST08,
  author    = {Eiseman, Phil and Lima, Jonathan D. and Sabloff, Joshua M. and Traynor, Lisa},
  title     = {A partial ordering on slices of planar {L}agrangians},
  journal   = {Journal of Fixed Point Theory and Applications},
  volume    = {3},
  number    = {2},
  pages     = {431--447},
  year      = {2008},
  publisher = {Springer},
  doi       = {10.1007/s11784-008-0087-0}
}

@article{ng2003computable,
  author    = {Ng, Lenhard},
  title         = {Computable Legendrian invariants},
  journal       = {Topology},
  volume        = {42},
  number        = {1},
  pages         = {55--82},
  year          = {2003},
  publisher = {Mathematical Sciences Publishers},
  doi       = {10.1016/S0040-9383(02)00010-1}
}

@article{rizell2024Lagrangian,
  author  = {Dimitroglou Rizell, Georgios},
  title   = {Lagrangian approximation of totally real concordances},
  journal = {Proceedings of the London Mathematical Society},
  volume  = {130},
  pages   = {e70042},
  year    = {2025}
}

@incollection{egh2000sft,
  author    = {Eliashberg, Yakov and Givental, Alexander and Hofer, Helmut},
  title     = {Introduction to symplectic field theory},
  booktitle = {Visions in Mathematics},
  pages     = {560--673},
  year      = {2000},
  publisher = {Birkh{\"a}user},
  doi       = {10.1007/978-3-0346-0425-3_4},
  eprint    = {math/0010059},
  archivePrefix = {arXiv},
  note      = {Special volume, GAFA 2000, Part II}
}

@article{arn1,
  author   = {Arnol'd, Vladimir I.},
  title    = {Lagrange and {L}egendre cobordisms. {I}},
  journal  = {Functional Analysis and Its Applications},
  volume   = {14},
  number   = {3},
  pages    = {167--177},
  year     = {1980},
  doi      = {10.1007/BF01086095}
}

@article{arn2,
  author   = {Arnol'd, Vladimir I.},
  title    = {Lagrange and {L}egendre cobordisms. {II}},
  journal  = {Functional Analysis and Its Applications},
  volume   = {14},
  number   = {4},
  pages    = {252--260},
  year     = {1980},
  doi      = {10.1007/BF01078301}
}

@article{Sau04,
  author  = {Sauvaget, Denis},
  title   = {Curiosit{\'e}s lagrangiennes en dimension 4},
  journal = {Annales de l'Institut Fourier (Grenoble)},
  volume  = {54},
  number  = {6},
  pages   = {1997--2020},
  year    = {2004}
}

@article{Chantraine2015sym,
  author    = {Baptiste Chantraine},
  title     = {Lagrangian concordance is not a symmetric relation},
  journal   = {Quantum Topology},
  volume    = {6},
  number    = {3},
  pages     = {451--474},
  year      = {2015},
  doi       = {10.4171/QT/68},
  eprint    = {1301.3767},
  archivePrefix = {arXiv}
}
\bibliographystyle{plain}
\end{document}